\documentclass[twocolumn]{autart}    

\usepackage{graphicx}          

\usepackage{mathrsfs}
\usepackage{graphics} 
\usepackage{amsmath} 
\usepackage{amssymb}  

\usepackage{savesym}
\savesymbol{theoremstyle}
\usepackage{amsthm}
\restoresymbol{AMS}{theoremstyle}
\usepackage{cite}
\usepackage{bm}
\usepackage{acronym}
\usepackage{paralist}
\usepackage{float}
\usepackage{color}
\usepackage{epstopdf}
\usepackage{multicol}
\usepackage{tikz}
\usepackage{hyperref}
\usepackage{soul}
\usepackage{mathtools}
\usepackage{siunitx}
\usepackage{empheq}
\usepackage{accents}
\usepackage{subcaption}
\usepackage{mdframed}
\usepackage[export]{adjustbox}
\usepackage{mathptmx}
\usepackage{anyfontsize}
\usepackage{enumitem}
\usepackage{xr}

\makeatletter
\def\thm@space@setup{%
  \thm@preskip=2mm
  \thm@postskip=\thm@preskip 
}
\makeatother
\newtheorem{Theorem}{Theorem}
\newtheorem{Lemma}{Lemma}
\newtheorem{Problem}{Problem}
\newtheorem{Remark}{Remark}
\newtheorem{Corollary}{Corollary}

\newtheorem{Assumption}{Assumption}
\newtheorem{Definition}{Definition}

\DeclareMathOperator*{\argmin}{arg\,min}

\begin{document}

\begin{frontmatter}

\title{Fixed-Time Parameter Adaptation for Safe Control Synthesis\thanksref{footnoteinfo}} 
                                                

\thanks[footnoteinfo]{This paper was not presented at any IFAC 
meeting. Corresponding author is Mitchell Black.}

\author[Michigan]{Mitchell Black}\ead{mblackjr@umich.edu},
\author[Ford]{Ehsan Arabi}\ead{earabi@ford.com},
\author[Michigan2]{Dimitra Panagou}\ead{dpanagou@umich.edu}    

\address[Michigan]{Department of Aerospace Engineering, University of Michigan, Ann Arbor, MI 48109, USA}
\address[Ford]{Ford Research and Advanced Engineering, Ford Motor Company, Dearborn, MI 48124, USA}
\address[Michigan2]{Department of Robotics and Department of Aerospace Engineering, University of Michigan, Ann Arbor, MI 48109, USA}

\begin{keyword}                           
control of constrained systems; robust control of nonlinear systems; safety-critical systems; fixed-time stability. 
\end{keyword}                             

\begin{abstract}                          
This paper introduces a parameter adaptation-based control law for a class of nonlinear, control-affine, safety-critical systems subject to additive, parameter-affine model uncertainty. It is shown that the uncertainty is learned in fixed-time, i.e., within a finite time independent of the initial parameter estimates, without requiring persistence of excitation or rank conditions on the associated regressor matrix. A monotonically decreasing bound on the error of the estimated uncertainty is derived and used in the proposed control barrier function-based control law for guaranteed safety of the system trajectories. It is then proven that the parameter adaptation law is robust to bounded measurement noise and non-parametric perturbations to the system in that the estimated uncertainty still converges to a known neighborhood of the true uncertainty in fixed-time. The advantages of the approach are highlighted in a comparative numerical study, and the method is validated on a quadrotor trajectory-tracking problem subject to safety requirements and an unknown wind field.
\end{abstract}

\end{frontmatter}

\section{Introduction}
Synthesizing safe controllers for dynamical systems under uncertainty is challenging. The rise of control barrier functions (CBFs) \cite{ames2017control}, 
a model-based tool for synthesizing safety with stabilization objectives online via quadratic programs (QPs), has advanced the state-of-the art in the field. However, inaccurate system models may cause safety violations in CBF-based control. Typically, model uncertainty is compensated for either by 1) taking a robust approach \cite{xu2015robustness,jankovic2018robustcbf}, assuming a worst-case bound on system perturbations; by 2) using adaptive techniques \cite{Taylor2019aCBF,Lopez2020racbf,Zhao2020robustQP}, updating estimates of unknown system parameters online; or by 3) employing learning algorithms \cite{taylor2020learning,Jagtap2020unknown}, using data to learn a representation of system behavior. Categorically, these approaches possess both strengths and weaknesses: robust controllers protect against bounded perturbations in exchange for conservatism; adaptation-based parameter estimators adjust to changes in the system dynamics online, but often require persistently excited (PE) signals; and while learning algorithms have found empirical success, they often lack guarantees of safety or stabilization. This paper proposes a control framework that leverages the reactivity of adaptation and the protection of robust control to guarantee that 1) a class of additive, parameter-affine perturbations to the system dynamics is learned within a finite time, and 2) the system trajectories remain safe at all times despite model uncertainty. Integral to the proposed framework is the notion of fixed-time stability.

While asymptotic (AS) and exponential (ES) stability imply convergence of the trajectories of a dynamical system to an equilibrium as time tends to infinity, fixed-time stability (FxTS) \cite{polyakov2012nonlinear} guarantees convergence to the equilibrium \textit{within a finite time, independent of the initial conditions}. 
Much recent work has been dedicated to FxTS for control design \cite{Parsegov2012Fixed,polyakov2023finite}, estimation \cite{Rios2017TimeVarying,ochoa2021accelerated,tatari2021fixed}, and optimization \cite{poveda2021fixed,garg2020fixed}, though little that utilizes adaptation for control of safety-critical systems. For example, 
\cite{Tao2022FilteredAdaptive} and \cite{Rios2017TimeVarying} design FxTS parameter estimators for 
neural networks and perturbed linear regression systems respectively, but control design is not addressed.
Recently, the dynamic regressor extension and mixing algorithm has received significant attention in the field of parameter estimation for its promise of monotonicity of each element of the parameter error vector without the often restrictive persistence of excitation condition \cite{Aranovskiy2017DREM,Aranovskiy2019LTI}. As shown in \cite{Aranovskiy2019LTI}, however, it still requires the PE condition for global convergence of the parameter error.

Meanwhile, existing work on safe control for systems with additive, parameter-affine uncertainty has focused on safety at the exclusion of stabilization. In \cite{Taylor2019aCBF} the adaptive CBF is introduced for system-level safety, not for learning model uncertainty.
This open problem was addressed in \cite{Lopez2020racbf} by combining adaptive CBFs with a data-driven set membership identification (SMID)-based parameter update law, which, though leading to improved empirical results, lacks guarantees of convergence of the parameter estimates. To the best of our knowledge, the first design of a safety-critical fixed-time adaptive control law was proposed in our prior work \cite{Black2021FixedTime}, though the robustness of the approach to measurement error and unmodelled disturbances was not investigated.


In this paper, a fixed-time adaptive control law is introduced to reduce conservatism in control of nonlinear, control- and parameter-affine safety-critical systems. This is achieved through a novel parameter adaptation law derived to identify the parameter-affine perturbation to the system in fixed-time, notably without requiring the PE condition. The robustness of the proposed adaptation law with respect to bounded measurement error or exogenous perturbations is analyzed, and a fixed-time domain of attraction is derived. Then, the monotonically-decreasing bound on the error of the perturbation estimation is used to synthesize a CBF condition that is both robust to present levels of uncertainty and adaptive to the changing parameter estimates. It is shown how the satisfaction of this condition via control design guarantees safety at all times. The advantages of the proposed method are highlighted in comparison with related works from the literature on a single-integrator example, and are further demonstrated on a simulated quadrotor system in unknown wind fields.


The paper is organized as follows. Section \ref{sec: math prelim} reviews preliminaries and introduces the problem. Section \ref{sec: FxT Parameter Estimation} contains the main results on FxTS parameter adaptation, while Section \ref{sec: Robust Adaptive Safety} analyzes their robustness against a class of bounded perturbations and proposes the new CBF condition. Section \ref{sec: simulation} contains the results of the case studies, and Section \ref{sec: conclusion} concludes with summaries and directions for future work.

\section{Mathematical Preliminaries}\label{sec: math prelim}

The following notation is used. $\mathbb R$ denotes the set of real numbers. The ones matrix of size $n \times m$ is denoted $\boldsymbol{1}_{n \times m}$. $\|\cdot\|$ is the Euclidean norm and $\|\cdot\|_{\infty}$ the supremum norm. The minimum, maximum, and r$^{th}$ singular values of a matrix $\boldsymbol{M} \in \mathbb R^{n \times m}$ of column-rank $r \leq m$ are denoted $\sigma_{min}(\boldsymbol{M})$, $\sigma_{max}(\boldsymbol{M})$, and $\sigma_r(\boldsymbol{M})$. Further, $\underaccent{\bar}{\sigma}(\boldsymbol{M},T) \triangleq \inf_{t \leq T}\sigma_r(\boldsymbol{M}(t))$. The rowspace and nullspace are $\mathcal{R}(\boldsymbol{M})$ and $\mathcal{N}(\boldsymbol{M})$ respectively. When $n = m$ the minimum (resp. maximum) eigenvalue is $\lambda_{min}(\boldsymbol{M})$ (resp. $\lambda_{max}(\boldsymbol{M})$). 
The boundary of a closed set $S$ is denoted $\partial S$, and $\textrm{int}(S)$ is the interior. The Lie derivative of a continuously differentiable function $V:\mathbb R^n\mapsto \mathbb R$ along a vector field $f:\mathbb R^n\mapsto\mathbb R^n$ at a point $x\in \mathbb R^n$ is denoted by $L_fV(x) \triangleq \frac{\partial V}{\partial x} f(x)$.

Consider the following class of nonlinear, control- and parameter-affine systems:
\begin{equation}\label{uncertain system}
\dot{\boldsymbol{x}} = f(\boldsymbol{x}) + g(\boldsymbol{x})\boldsymbol{u} + \Delta(\boldsymbol{x}) \boldsymbol{\theta}^*, \; \boldsymbol{x}(0) = \boldsymbol{x}_0,
\end{equation}
where $\boldsymbol{x} \in \mathbb{R}^n$ denotes the state, $\boldsymbol{u} \in \mathcal{U} \subseteq \mathbb{R}^m$ the control input, and $\boldsymbol{\theta}^* \in \Theta \subset \mathbb{R}^p$ a vector of unknown, static parameters. The sets $\mathcal{U}$ and $\Theta$ are the input constraint and admissible parameter sets respectively, where both are known and $\Theta$ is a polytope. The fields $f: \mathbb{R}^n \mapsto \mathbb{R}^n$ and $g: \mathbb{R}^n \mapsto \mathbb{R}^{n \times m}$, and regressor $\Delta : \mathbb{R}^n \mapsto \mathbb{R}^{n \times p}$ are known and continuous such that for a continuous control input \eqref{uncertain system} admits a unique solution. The term $d(\boldsymbol{x}) = \Delta(\boldsymbol{x})\boldsymbol{\theta}^*$ may describe, e.g., modelling errors or disturbances in the system dynamics. Note that nothing is assumed about the rank of $\Delta(\boldsymbol{x})$, and thus there may be other vectors $\boldsymbol{\theta}$ that satisfy $\Delta(\boldsymbol{x})\boldsymbol{\theta}=d(\boldsymbol{x})$. As such, define the following state- and parameter-dependent set, $\Omega(\boldsymbol{x},\boldsymbol{\theta}^*)$, containing such vectors $\boldsymbol{\theta}$:
\begin{equation}\label{set: Omega}
    \Omega(\boldsymbol{x},\boldsymbol{\theta}^*) = \{\boldsymbol{\theta} \in \Theta: \Delta(\boldsymbol{x})\boldsymbol{\theta} = \Delta(\boldsymbol{x})\boldsymbol{\theta}^*\}.
\end{equation}
In the remainder, $\Omega$ is used for conciseness. The estimated parameter vector is $\hat{\boldsymbol{\theta}}$ so that the parameter estimation error vector is $\Tilde{\boldsymbol{\theta}} = \boldsymbol{\theta}^* - \hat{\boldsymbol{\theta}}$. Note that $\hat{\boldsymbol{\theta}} \in \Omega$ whenever $\Tilde{\boldsymbol{\theta}} \in \mathcal{N}(\Delta(\boldsymbol{x}))$.
The parameter estimation error dynamics are
\begin{equation}\label{error dynamics}
    \dot{\Tilde{\boldsymbol{\theta}}} = - \dot{\hat{\boldsymbol{\theta}}}, \;\; \Tilde{\boldsymbol{\theta}}(0) = \Tilde{\boldsymbol{\theta}}_0.
\end{equation}
With $\Theta$ a polytope, for any $\hat{\boldsymbol{\theta}} \in \Theta$ it follows that $\|\Tilde{\boldsymbol{\theta}}\|_{\infty} \leq \vartheta(t)$, where
$\vartheta(t) \triangleq \sup_{\boldsymbol{\theta}_1,\boldsymbol{\theta}_2 \in \Lambda(t)}(\|\boldsymbol{\theta}_1 - \boldsymbol{\theta}_2\|_{\infty})$ and $\Lambda(t) \subseteq \Theta$ denotes the set of admissible parameters at time $t$ (which may shrink as measurements are gathered).

Let $\boldsymbol{\vartheta} = \vartheta \cdot \boldsymbol{1}_{p \times 1}$, 
and consider a set of states defined with respect to a constraint function $h: \mathbb{R}^n \mapsto \mathbb{R}$ and positive-definite robustness margin $m: \mathbb R^p \mapsto \mathbb R_+$, both continuously differentiable:
\begin{equation}\label{eq.safe_set}
    S = \left \{ \boldsymbol{x} \in \mathbb{R}^n : h(\boldsymbol{x}) \geq m(\boldsymbol{\vartheta}) \right \}.
\end{equation}
Let $m(\boldsymbol{\vartheta}) \triangleq \frac{1}{2} \boldsymbol{\vartheta}^T \boldsymbol\Gamma ^{-1} \boldsymbol{\vartheta}$ for a constant, positive-definite matrix $\boldsymbol{\Gamma}$ such that $h(\boldsymbol{x}_0) \geq m(\boldsymbol{\vartheta})$, and assume that $\frac{\partial h}{\partial \boldsymbol{x}} \neq 0$, $\forall \boldsymbol{x} \in \partial S$. The following, known as Nagumo's Theorem \cite{BLANCHINI1999SetInvariance}, provides a necessary and sufficient condition for the forward invariance of the set \eqref{eq.safe_set} under the dynamics \eqref{uncertain system}.
\begin{Lemma}
    The set $S$ is forward-invariant for \eqref{uncertain system} if and only if $L_fh(\boldsymbol{x}) + L_gh(\boldsymbol{x})\boldsymbol{u} + L_{\Delta}h(\boldsymbol{x})\boldsymbol{\theta}^* - \dot m(\boldsymbol{\vartheta})\geq 0$ for all $\boldsymbol{x} \in \partial S$.
\end{Lemma}
Let $F(\boldsymbol{x}) = f(\boldsymbol{x}) + \Delta(\boldsymbol{x})\boldsymbol{\theta}^*$. 
One way to render $S$ forward-invariant is to use CBFs for control design \cite{ames2017control}.
\begin{Definition}\label{def: cbf}
    Given a set $S \subset \mathbb R^n$ defined by \eqref{eq.safe_set} for a continuously differentiable function $h: \mathbb{R}^n \mapsto \mathbb R$, the function $h$ is a \textbf{control barrier function} (CBF) defined on a set $D$, where $S \subseteq D \subset \mathbb R^n$, if there exists an extended class $\mathcal{K}_\infty$ function $\alpha: \mathbb R \mapsto \mathbb R$ such that, $\forall \boldsymbol{x}\in D$,
    \begin{equation}\label{eq: cbf condition}
        \sup_{\boldsymbol{u} \in \mathcal{U}}\big[L_Fh(\boldsymbol{x}) + L_gh(\boldsymbol{x})\boldsymbol{u}\big] \geq -\alpha\big(h(\boldsymbol{x}) - m(\boldsymbol{\vartheta})\big) + \dot m(\boldsymbol{\vartheta}).
    \end{equation}
\end{Definition}
Note that if $\dot m({\boldsymbol{\vartheta}}) = 0$ and $\boldsymbol{\vartheta} \neq 0$, then \eqref{eq: cbf condition} becomes $\dot{h} \geq -\alpha(h) + \gamma$ for $\gamma > 0$, a robust CBF condition \cite{jankovic2018robustcbf}.


Consider now the autonomous, nonlinear system
\begin{equation}\label{nonlinear system}
    \begin{aligned}
        \dot{\boldsymbol{x}} &= f(\boldsymbol{x}(t)), \quad \boldsymbol{x}(0) = \boldsymbol{x}_0,
    \end{aligned}
\end{equation}
where $\boldsymbol{x} \in \mathbb{R}^n$, $f:\mathbb{R}^n \mapsto \mathbb{R}^n$ is continuous such that \eqref{nonlinear system} admits a unique solution, and it is assumed without loss of generality that the origin is an equilibrium, i.e. $f(0) = 0$. 
\begin{Definition}\cite[Definition 2]{polyakov2012nonlinear}\label{FxTS Def}
    The origin of \eqref{nonlinear system} is \textbf{fixed-time stable} (FxTS) if it is
    stable in the sense of Lyapunov and any solution $\boldsymbol{x}(t, \boldsymbol{x}_0)$ reaches the origin with uniformly bounded settling-time $T(\boldsymbol{x}_0) \leq T_b < \infty$, $\forall \boldsymbol{x}_0 \in \mathbb R^n$.
\end{Definition}
\begin{Theorem}\cite[Lemma 1]{polyakov2012nonlinear}\label{Thm: FxTS Lyapunov Conditions} Suppose that there exists a continuous, positive definite, radially unbounded function $V: \mathbb{R}^n \mapsto \mathbb{R}_+$ for the system \eqref{nonlinear system} such that
\begin{align}\label{eq: FxTS Lyapunov conditions}
    \dot{V}(\boldsymbol{x}) \leq -c_1V^{\gamma_1}(\boldsymbol{x}) - c_2V^{\gamma_2}(\boldsymbol{x}),
\end{align}
for all $\boldsymbol{x} \in \mathbb R^n \setminus \{0\}$, where $c_1,c_2>0$, $\gamma_1>1$, and $0<\gamma_2<1$. Then, the origin of \eqref{nonlinear system} is fixed-time stable with settling time, $T(\boldsymbol{x}_0)$, bounded for all $\boldsymbol{x}_0 \in \mathbb R^n$ by
\begin{align}\label{eq: FxTS Convergence Time}
    T(\boldsymbol{x}_0) \leq T_b = \frac{1}{c_1(\gamma_1-1)} + \frac{1}{c_2(1-\gamma_2)}.
\end{align}
\end{Theorem}
It has been shown that learning $\boldsymbol{\theta}^*$ in fixed-time leads to favorable disturbance rejection properties \cite{Chang2021Fixed}.
For control design, however, identifying $\boldsymbol{\theta}^*$ \textit{exactly} is unnecessary when there are multiple $\boldsymbol{\theta} \in \Theta$ that satisfy $\Delta(\boldsymbol{x})\boldsymbol{\theta} = \Delta(\boldsymbol{x})\boldsymbol{\theta}^*$, as is allowed here. In this case, an estimate $\hat{\boldsymbol{\theta}} \in \Omega$ where $\Omega$ is given by \eqref{set: Omega} is sufficient for learning $\Delta(\boldsymbol{x})\boldsymbol{\theta}^*$. In this paper, it is shown how learning such an estimate within a fixed time adds value to a class of safe controllers. 
The problem under consideration is formalized as follows.
\begin{Problem}\label{Problem Statement}
    Consider a system of the form \eqref{uncertain system}. Design adaptation and control laws, $\dot{\hat{\boldsymbol{\theta}}} = \tau(\boldsymbol{x},\hat{\boldsymbol{\theta}})$ and $\boldsymbol{u}=k(\boldsymbol{x},\hat{\boldsymbol{\theta}})$, such that:
    \vspace{-3mm}
    \begin{enumerate}
        \item The estimated parameter vector, $\hat{\boldsymbol{\theta}}$, is rendered fixed-time stable to the set $\Omega$ given by \eqref{set: Omega}, i.e. $\hat{\boldsymbol{\theta}}(t) \in \Theta$, $\forall t \geq 0$ and $\hat{\boldsymbol{\theta}}(t) \rightarrow \Omega$ as $t \mapsto T$, 
        independent of $\hat{\boldsymbol{\theta}}(0)$,
        \item The system trajectories remain safe for all time, i.e. $\boldsymbol{x}(t) \in S$,  $\forall t \geq 0$.
    \end{enumerate}
\end{Problem}

\section{Parameter Adaptation}\label{sec: FxT Parameter Estimation}

The following is a common, albeit restrictive, assumption in the parameter identification literature (e.g., \cite{Na2015robotic,Rios2017TimeVarying}). 
\begin{Assumption}\label{ass: v=M0_ES}
    For all time $t \geq 0$ there is a known vector, $\boldsymbol{v}(t) \in \mathbb R^q$, and a known, full column rank matrix, $\boldsymbol{M}(t) \in \mathbb R^{q \times p}$,
    such that $\Tilde{\boldsymbol{\theta}}(t)$ is a solution to 
    \begin{equation}\label{eq.es_matrix}
        \boldsymbol{M}(t)\Tilde{\boldsymbol{\theta}}(t)=\boldsymbol{v}(t).
    \end{equation}
\end{Assumption}
\noindent The above more commonly appears in the form $\sigma_{min}(\boldsymbol{M}(t)) > s > 0$, $\forall t \geq 0$, and is equivalent to system identifiability \cite[Def. 4.6]{Ljung1987SystemID}. Its utility is in the design of adaptation laws that provide stability guarantees (e.g., ES \cite{Na2015robotic} or FTS/FxTS \cite{Rios2017TimeVarying}) when \eqref{eq.es_matrix} is perturbed, i.e., when
\begin{equation}\label{eq.perturbed_mv}
    \boldsymbol{M}(t)\Tilde{\boldsymbol{\theta}}(t)=\boldsymbol{v}(t)-\boldsymbol{\delta}(t),
\end{equation}
for some unknown $\boldsymbol{\delta} \in \mathbb{R}^q$. 
The authors of \cite{Na2015robotic} further note that Assumption \ref{ass: v=M0_ES} can be satisfied when the regressor, $\Delta$, of system \eqref{uncertain system} satisfies the PE condition. The following represents a relaxation of Assumption \ref{ass: v=M0_ES}.

\begin{Assumption}\label{ass: v=M0}
    For all $t \geq 0$ there is a known vector, $\boldsymbol{v}(t) \in \mathbb R^n$, and a known matrix, $\boldsymbol{M}(t) \in \mathbb R^{n \times p}$, that jointly satisfy
    \vspace{-3mm}
    \begin{enumerate}
        \item $\Tilde{\boldsymbol{\theta}}(t)$ is one (not necessarily unique) solution to \eqref{eq.es_matrix} \label{ass.vm0_1}
        \item $\boldsymbol{v}(t)$ and $\boldsymbol{M}(t)$ are bounded for bounded $\boldsymbol{x}(t)$
        \item $\mathcal{N}(\boldsymbol{M}(t)) = \mathcal{N}(\Delta(\boldsymbol{x}(t)))$.\label{ass.vm0_2}
    \end{enumerate}
\end{Assumption}
\begin{Remark}
    In contrast to Assumption \ref{ass: v=M0_ES}, Assumption \ref{ass: v=M0} does not require that $\boldsymbol{M}$ be full column-rank. Instead, it requires that $\mathcal{N}(\boldsymbol{M}) = \mathcal{N}(\boldsymbol{\Delta})$, which relaxes restrictions imposed by previous works \cite{Na2015robotic,Rios2017TimeVarying}.
\end{Remark}
\begin{Remark}\label{rmk.schemes}
    Note that Assumption \ref{ass: v=M0} is satisfied with $\boldsymbol{M}$ and $\boldsymbol{v}$ obtained by rewriting the system dynamics \eqref{uncertain system} as
\begin{align}\label{eq: rate v=M0} \underbrace{\Delta(\boldsymbol{x})}_\text{$\boldsymbol{M}$}\Tilde{\boldsymbol{\theta}}=\underbrace{\dot{\boldsymbol{x}} - f(\boldsymbol{x}) - g(\boldsymbol{x})\boldsymbol{u} - \Delta(\boldsymbol{x})\hat{\boldsymbol{\theta}}}_\text{$\boldsymbol{v}$}.
\end{align}
Clearly, \eqref{eq: rate v=M0} requires that the state and its time-derivative are available. When measurements of $\dot{\boldsymbol{x}}$ are not available, they can be approximated using the state measurements and numerical differentiation. Such schemes, however, can introduce approximation error into \eqref{eq: rate v=M0} even under perfect state measurements, so that it takes the form \eqref{eq.perturbed_mv}. Such errors  along with the consideration that the state measurements are in general imperfect are addressed in Section \ref{sec: Robust Adaptive Safety}. While there are other viable approaches, including the observer introduced by \cite{Adetola2008ParameterEstimation}, the details are beyond the scope of this paper.
\end{Remark}
The following is one of the main results of the paper, an adaptation law that renders $\hat{\boldsymbol{\theta}}$ FxTS to the set \eqref{set: Omega}.
\begin{Theorem}\label{Thm: FxT Parameter Adaptation}
    Consider a dynamical system of the form \eqref{uncertain system} and suppose that Assumption \ref{ass: v=M0} holds. If  $\mathcal{N}(\Delta(\boldsymbol{x}(t))) = \mathcal{N}(\Delta(\boldsymbol{x}(0)))$, $\forall t \in [0,T]$, where
    \begin{equation}\label{eq: FxT Time Bound}
        T = \frac{\mu\pi}{2k_V^2\sqrt{ab}},
    \end{equation}
    with $a,b>0$, $\mu > 2$, and
    \begin{equation}\label{eq: kV}
        k_V = \underaccent{\bar}{\sigma}(\boldsymbol{M},T)\sqrt{2\lambda_{max}(\boldsymbol{\Gamma})},
    \end{equation}
    where $\boldsymbol\Gamma \in \mathbb R^{p\times p}$ is a constant, positive-definite, gain matrix, 
    then under the ensuing adaptation law,
    \begin{equation}\label{eq: FxT Estimation Law}
        \dot{\hat{\boldsymbol{\theta}}} = \boldsymbol\Gamma \boldsymbol{M}^T\boldsymbol{v}\left(a\|\boldsymbol{v}\|^{ \frac{2}{\mu}} + b\|\boldsymbol{v}\|^{-\frac{2}{\mu}}\right),
    \end{equation}
    the estimated parameter vector, $\hat{\boldsymbol{\theta}}$, converges to $\Omega$ defined by \eqref{set: Omega} within a fixed time $T$, i.e. $\hat{\boldsymbol{\theta}}(t) \rightarrow \Omega$ and $\Tilde{\boldsymbol{\theta}}(t) \rightarrow \mathcal{N}(\Delta(\boldsymbol{x}(t))$ as $t \rightarrow T_c(\hat{\boldsymbol{\theta}}(0)) \leq T$, $\forall \hat{\boldsymbol{\theta}}(0) \in \mathbb R^p$.
\end{Theorem}
\begin{proof}
It will first be shown that \eqref{eq: FxT Estimation Law} is well-defined, and then proven that $\hat{\boldsymbol{\theta}}$ converges to $\Omega$ within fixed-time $T$. 

It is obvious by boundedness of $\boldsymbol{v}$ and $\boldsymbol{M}$ that \eqref{eq: FxT Estimation Law} is bounded whenever $0<\|\boldsymbol{v}\|<\infty$, so consider when $\|\boldsymbol{v}\|\rightarrow 0$. Note that \eqref{eq: FxT Estimation Law} can be rewritten as $\dot{\hat{\boldsymbol{\theta}}} = G_1(\boldsymbol{M},\boldsymbol{v}) + G_2(\boldsymbol{M},\boldsymbol{v})$, where $G_1(\boldsymbol{M},\boldsymbol{v}) = a\boldsymbol\Gamma \boldsymbol{M}^T\boldsymbol{v}\|\boldsymbol{v}\|^{\frac{2}{\mu}}$ and $G_2(\boldsymbol{M},\boldsymbol{v}) = b\boldsymbol\Gamma \boldsymbol{M}^T\boldsymbol{v}\|\boldsymbol{v}\|^{-\frac{2}{\mu}}$. It is evident that $G_1(\boldsymbol{M},\boldsymbol{v}) \rightarrow 0$ as $\|\boldsymbol{v}\|\rightarrow 0$, however, for $G_2(\boldsymbol{M},\boldsymbol{v})$ it follows that $\|G_2(\boldsymbol{M},\boldsymbol{v})\|^2 = b^2\|\boldsymbol{v}\|^{-4/\mu}\boldsymbol{v}^T\left(\boldsymbol{M}\boldsymbol\Gamma^T\boldsymbol\Gamma \boldsymbol{M}^T\right)\boldsymbol{v}$.
Then, defining $\boldsymbol{N} = \boldsymbol{M}\boldsymbol\Gamma^T\boldsymbol\Gamma \boldsymbol{M}^T$, it holds that, $\forall t \leq T$, 
\begin{align}
    b^2\sigma_{min}(\boldsymbol{N})\|\boldsymbol{v}\|^{\phi} \leq \|G_2(\boldsymbol{M},\boldsymbol{v})\|^2 \leq b^2\sigma_{max}(\boldsymbol{N})\|\boldsymbol{v}\|^{\phi} \nonumber,
\end{align}
with $\phi = 2 - 4/\mu$.
Since $0 \leq \sigma_{min}(\boldsymbol{N}),\sigma_{max}(\boldsymbol{N}) < \infty$, it follows that if $\mu > 2$, then $\lim_{\|\boldsymbol{v}\|\rightarrow 0}\|G_2(\boldsymbol{M},\boldsymbol{v})\| = 0$.
Thus, with $\mu > 2$,
\eqref{eq: FxT Estimation Law} is well-defined.

For convergence of $\hat{\boldsymbol{\theta}}$ to $\Omega$, observe that $\Tilde{\boldsymbol{\theta}}$, $\dot{\Tilde{\boldsymbol{\theta}}}$, and $\dot{\hat{\boldsymbol{\theta}}}$ may be expressed as the following linear combinations: $\Tilde{\boldsymbol{\theta}} = \Tilde{\boldsymbol{\theta}}_R + \Tilde{\boldsymbol{\theta}}_N$,  $\dot{\Tilde{\boldsymbol{\theta}}} = \dot{\Tilde{\boldsymbol{\theta}}}_R + \dot{\Tilde{\boldsymbol{\theta}}}_N$, and $\dot{\hat{\boldsymbol{\theta}}} = \dot{\hat{\boldsymbol{\theta}}}_R + \dot{\hat{\boldsymbol{\theta}}}_N$,
where $\Tilde{\boldsymbol{\theta}}_R,\dot{\Tilde{\boldsymbol{\theta}}}_R,\dot{\hat{\boldsymbol{\theta}}}_R \in \mathcal{R}(\boldsymbol{M})$ and $\Tilde{\boldsymbol{\theta}}_N,\dot{\Tilde{\boldsymbol{\theta}}}_N,\dot{\hat{\boldsymbol{\theta}}}_N \in \mathcal{N}(\boldsymbol{M})$. By rank-nullity and with row$(\boldsymbol{M})$ and $\mathcal{N}(\boldsymbol{M})$ as orthogonal complements, such a decomposition always exists \cite{Strang2016LinAlg}. Next, from \eqref{eq: FxT Estimation Law} and Assumption \ref{ass: v=M0}, 
\begin{align}
    \dot{\hat{\boldsymbol{\theta}}} &= \boldsymbol\Gamma \boldsymbol{M}^T\boldsymbol{v}\left(a\|\boldsymbol{v}\|^{ \frac{2}{\mu}} + b\|\boldsymbol{v}\|^{-\frac{2}{\mu}}\right) \nonumber \\
    &= \boldsymbol\Gamma \boldsymbol{M}^T\boldsymbol{M}\Tilde{\boldsymbol{\theta}}\left(a\|\boldsymbol{M}\Tilde{\boldsymbol{\theta}}\|^{\frac{2}{\mu}} + b\|\boldsymbol{M}\Tilde{\boldsymbol{\theta}}\|^{-\frac{2}{\mu}}\right) \nonumber.
\end{align}
Since $\boldsymbol{M}\Tilde{\boldsymbol{\theta}}=\boldsymbol{M}(\Tilde{\boldsymbol{\theta}}_R+\Tilde{\boldsymbol{\theta}}_N)=\boldsymbol{M}\Tilde{\boldsymbol{\theta}}_R$, it follows that $\dot{\hat{\boldsymbol{\theta}}} = \boldsymbol\Gamma \boldsymbol{M}^T\boldsymbol{M}\Tilde{\boldsymbol{\theta}}_R\left(a\|\boldsymbol{M}\Tilde{\boldsymbol{\theta}}_R\|^{\frac{2}{\mu}} + b\|\boldsymbol{M}\Tilde{\boldsymbol{\theta}}_R\|^{ \frac{-2}{\mu}}\right) = \dot{\hat{\boldsymbol{\theta}}}_R$.
Then, from $\dot{\hat{\boldsymbol{\theta}}} = \dot{\hat{\boldsymbol{\theta}}}_R + \dot{\hat{\boldsymbol{\theta}}}_N$, it holds that $\dot{\hat{\boldsymbol{\theta}}}_N = \boldsymbol{0}_{p \times 1}$ and thus by \eqref{error dynamics}, 
\begin{equation}
    \dot{\Tilde{\boldsymbol{\theta}}}_R = -\dot{\hat{\boldsymbol{\theta}}}_R, \;\; \Tilde{\boldsymbol{\theta}}_R(0) = \Tilde{\boldsymbol{\theta}}_{R,0}. \label{eq.err_tilde_R} 
\end{equation}
By Assumption \ref{ass: v=M0}, $\Tilde{\boldsymbol{\theta}}_N(0) \in \mathcal{N}(\Delta(\boldsymbol{x}(0)))$, and by $\mathcal{N}(\Delta(\boldsymbol{x}(t)))=\mathcal{N}(\Delta(\boldsymbol{x}(0)))$ for all $t \in [0,T]$, $\hat{\boldsymbol{\theta}} \in \Omega$ whenever $\Tilde{\boldsymbol{\theta}}_R = 0$. It is sufficient, therefore, to show that $\Tilde{\boldsymbol{\theta}}_R(t) \rightarrow 0$ as $t \rightarrow T_c(\hat{\boldsymbol{\theta}}(0)) \leq T$.

Consider the Lyapunov function candidate $V = \frac{1}{2} \Tilde{\boldsymbol{\theta}}_R^T \boldsymbol{\Gamma}^{-1} \Tilde{\boldsymbol{\theta}}_R$. Its time derivative along the trajectories of \eqref{eq.err_tilde_R} is
\begin{align}
    \dot{V} &= -\Tilde{\boldsymbol{\theta}}_R^T\boldsymbol{\Gamma}^{-1}\dot{\hat{\boldsymbol{\theta}}}_R \nonumber \\
    &= -\Tilde{\boldsymbol{\theta}}_R^T\boldsymbol{M}^T\boldsymbol{M}\Tilde{\boldsymbol{\theta}}_R\left(a\|\boldsymbol{M}\Tilde{\boldsymbol{\theta}}_R\|^{\frac{2}{\mu}} + b\|\boldsymbol{M}\Tilde{\boldsymbol{\theta}}_R\|^{-\frac{2}{\mu}}\right) \nonumber \\
    &= -a\|\boldsymbol{M}\Tilde{\boldsymbol{\theta}}_R\|^{2+\frac{2}{\mu}} - b\|\boldsymbol{M}\Tilde{\boldsymbol{\theta}}_R\|^{2-\frac{2}{\mu}}. \nonumber
\end{align}
When $\Tilde{\boldsymbol{\theta}}_R \neq 0$, we have that $\|\boldsymbol{M}\Tilde{\boldsymbol{\theta}}_R\| > 0$, and therefore $0 < \underaccent{\bar}{\sigma}(\boldsymbol{M}, T)\|\Tilde{\boldsymbol{\theta}}_R\| \leq \|\boldsymbol{M}\Tilde{\boldsymbol{\theta}}_R\|$, $\forall t \leq T$. Then, since $\frac{1}{2}\lambda_{min}(\boldsymbol{\Gamma}^{-1})\Tilde{\boldsymbol{\theta}}_R^T\Tilde{\boldsymbol{\theta}}_R \leq V$ and $\lambda_{min}(\boldsymbol{\Gamma}^{-1})=1/\lambda_{max}(\boldsymbol{\Gamma})$, it follows from the definition of $V$ that $\|\Tilde{\boldsymbol{\theta}}_R\| \leq \sqrt{2V\lambda_{max}(\boldsymbol{\Gamma})} = L(V)$. Using this, we obtain
\begin{align}
    \dot{V} &\leq -a\Big(\underaccent{\bar}{\sigma}(\boldsymbol{M},T)L(V)\Big)^{2+\frac{2}{\mu}} -b\Big(\underaccent{\bar}{\sigma}(\boldsymbol{M},T)L(V)\Big)^{2-\frac{2}{\mu}} \nonumber \\
    &\leq -c_1V^{1 + \frac{1}{\mu}} - c_2V^{1 - \frac{1}{\mu}}, \label{eq.vdot_proof}
\end{align}
which takes the form \eqref{eq: FxTS Lyapunov conditions}, where $c_1 = ak_V^{2+\frac{2}{\mu}}$ and $c_2 = bk_V^{2 - \frac{2}{\mu}}$,
and $k_V$ is given by \eqref{eq: kV}. Then, \eqref{eq: FxTS Lyapunov conditions} 
is met insofar as $\Tilde{\boldsymbol{\theta}}_R$ is concerned. Thus, $\Tilde{\boldsymbol{\theta}}_R \rightarrow 0$ as $t \rightarrow T_c(\hat{\boldsymbol{\theta}}(0))$, with $T_c$ the settling time function.
It then follows from \cite[Lem. 2]{Parsegov2012Fixed} that $T_c(\hat{\boldsymbol{\theta}}(0)) \leq T$, $\forall \hat{\boldsymbol{\theta}}(0) \in \mathbb R^p$, with $T$ given by \eqref{eq: FxT Time Bound}.
Thus, $\hat{\boldsymbol{\theta}} \rightarrow \Omega$ as $t \rightarrow T$, independent of the initial estimates, $\hat{\boldsymbol{\theta}}(0)$.
\end{proof}

\begin{Remark}\label{rmk: min singular val}
    The bound on settling time \eqref{eq: FxT Time Bound} requires knowledge of $\underaccent{\bar}{\sigma}(\boldsymbol{M},T)$. While this may not be known a priori, domain knowledge can provide a lower bound over the relevant region of the state space. In some cases, a warm-up time may be required for $\sigma_r(\boldsymbol{M}(t))$ to eclipse a lower bound. This must be added to the convergence time derived in \eqref{eq: FxT Time Bound}.
\end{Remark}
With knowledge of the parameter estimates converging to $\Omega$ in fixed-time, an expression for the upper bound on the supremum norm of the parameter error is now introduced.
\begin{Corollary}\label{Cor: upper bound parameter error}
Suppose that the premises of Theorem \ref{Thm: FxT Parameter Adaptation} hold. 
If, in addition, the following conditions are met:
\vspace{-2mm}
\begin{enumerate}[label=(\roman*)]
    \item The initial estimated parameter vector lies within the known admissible parameter set, $\Theta$, i.e. $\hat{\boldsymbol{\theta}}(0) \in \Theta$,
    \item The adaptation law, $\dot{\hat{\boldsymbol{\theta}}}$, is given by \eqref{eq: FxT Estimation Law}, where $\boldsymbol{\Gamma}$ is constant, positive-definite, and also \textbf{diagonal},
    \item A lower bound on $\sigma_r(\boldsymbol{M}) > s > 0$, is known,
\end{enumerate}
\vspace{-2mm}
then $\forall t \in [0,T]$, where $T$ is given by \eqref{eq: FxT Time Bound}, the following is a monotonically decreasing upper bound on $\|\Tilde{\boldsymbol{\theta}}_R(t)\|_{\infty}$:
\begin{equation}\label{eq.err_bound}
    \|\Tilde{\boldsymbol{\theta}}_R(t)\|_{\infty} \leq \sqrt{2 \lambda_{max}(\boldsymbol{\Gamma})\left( \sqrt{\frac{a}{b}}k_V^{\frac{2}{\mu}}\tan\left(A(t)\right) \right)^\mu} \coloneqq \eta (t),
\end{equation}
where $k_V$ is defined by \eqref{eq: kV}, and
\begin{align}
    A(t)&=\max\left\{\Xi - t\sqrt{c_1c_2}/\mu,0\right\}, \label{eq.A(t)} \\ 
    \Xi &= \tan^{-1}\left(\sqrt{\frac{c_2}{c_1}}\left(\frac{1}{2}\boldsymbol{\eta}(0)^T\Gamma^{-1}\boldsymbol{\eta}(0)\right)^{\frac{1}{\mu}}\right),
\end{align}
with $\boldsymbol{\eta}(t) = \eta(t) \cdot \boldsymbol{1}_{p\times1}$, and $\|\Tilde{\boldsymbol{\theta}}_R(t)\|_{\infty} =0$, $\forall t > T$.
\end{Corollary}
\begin{proof}
    Consider again the Lyapunov function candidate $V = \frac{1}{2}\Tilde{\boldsymbol{\theta}}_R ^T \Gamma ^{-1} \Tilde{\boldsymbol{\theta}}_R$, the time-derivative of which is given by \eqref{eq.vdot_proof}. Use a change of variables $x = V^{\frac{1}{\mu}}$ and $dx = \frac{1}{\mu}V^{\frac{1}{\mu}-1}dV$ to get
    \begin{align}
        t &\leq \int_{V_{0}^{\frac{1}{\mu}}}^{V^{\frac{1}{\mu}}} \frac{\mu x^{\mu - 1} dx}{-c_{1}x^{\mu-1}-c_{2}x^{\mu+1}} = \int_{V_{0}^{\frac{1}{\mu}}}^{V^{\frac{1}{\mu}}} \frac{\mu dx}{-c_{1}-c_{2}x^{2}}\nonumber \\
        &\leq-\frac{\mu}{\sqrt{c_{1}c_{2}}}\left (\tan^{-1}\left(\sqrt{\frac{c_2}{c_1}} V^{\frac{1}{\mu}}\right) -  \tan^{-1}\left(\sqrt{\frac{c_2}{c_1}} V_0^{\frac{1}{\mu}}\right) \right ) \nonumber.
    \end{align}
    By rearranging terms, 
    \begin{align}
        V(t) &\leq \left (\sqrt{\frac{c_1}{c_2}} \tan\left (\tan^{-1}\left(\sqrt{\frac{c_2}{c_1}} V_0^{\frac{1}{\mu}}\right) -\frac{\sqrt{c_1c_2}}{\mu}t \right)\right)^{\mu}, \nonumber
    \end{align}
    where, since $V_0 = \frac{1}{2}\Tilde{\boldsymbol{\theta}}_R(0)^T\Gamma^{-1}\Tilde{\boldsymbol{\theta}}_R(0) \leq \frac{1}{2}\boldsymbol{\eta}(0)^T\Gamma^{-1}\boldsymbol{\eta}(0)$, it follows that $V(t) \leq \left (\sqrt{\frac{c_1}{c_2}} \tan\left(\Xi -\frac{\sqrt{c_1c_2}}{\mu}t \right)\right)^{\mu}$.
    Observe that $\Xi \rightarrow \frac{\pi}{2}$ as $\|\boldsymbol{\eta}(0)\| \rightarrow \infty$. With $\hat{\boldsymbol{\theta}}(0) \in \Theta \subset \mathbb R^p$, however, it will not occur that $V_0 \rightarrow \infty$, and therefore, since \eqref{eq: FxT Time Bound} holds for arbitrarily large $V_0$, it is possible that $\Xi -\frac{\sqrt{c_1c_2}}{\mu}t < 0$ for some $t < T$. Thus, define $A(t)$ by \eqref{eq.A(t)} and obtain $V(t) \leq \left (\sqrt{c_1/c_2} \tan\left(A(t)\right)\right)^{\mu}$.
    Now, since $\Gamma$ is diagonal, $V = \frac{1}{2}(\Gamma^{-1}_{11}\Tilde{\theta}_{R,1}^2 + \cdots + \Gamma^{-1}_{pp}\Tilde{\theta}_{R,p}^2)$, and observe that $V \geq \frac{1}{2}\lambda^{-1}_{max}(\Gamma)\|\Tilde{\boldsymbol{\theta}}_R\|^2  \geq \frac{1}{2}\lambda^{-1}_{max}(\Gamma)\|\Tilde{\boldsymbol{\theta}}_R\|_{\infty}^2 $. Then, substitute $V(t)$ in this inequality and rearrange terms to recover \eqref{eq.err_bound}. Then, for $0 \leq t \leq \frac{\mu}{\sqrt{c_1c_2}}\textrm{tan}^{-1}\left(\sqrt{\frac{c_2}{c_1}}\left(\frac{1}{2}\boldsymbol{\eta}(0)^T\Gamma^{-1}\boldsymbol{\eta}(0)\right)^{\frac{1}{\mu}}\right)$
    we have that \eqref{eq.err_bound} decreases monotonically to zero. 
    This completes the proof.
\end{proof}

\begin{Remark}
    Recalling the definition of $\vartheta(t)$, from \eqref{eq.err_bound} define $\Lambda(t) = \{\boldsymbol{\lambda} \in \Theta : \|\boldsymbol{\lambda}_R - \hat{\boldsymbol{\theta}}_R\|_\infty \leq \eta(t)\}$ with $\boldsymbol{\lambda}=\boldsymbol{\lambda}_R + \boldsymbol{\lambda}_N$ for $\boldsymbol{\lambda}_R,\hat{\boldsymbol{\theta}}_R\in\mathcal{R}(\Delta)$ such that $\Lambda(t) \subset \Theta$ and $\boldsymbol{\theta}^* \in \Lambda(t)$. Note that $\Lambda(t) = \{\hat{\boldsymbol{\theta}} + \mathcal{N}(\Delta)\}$ for all $t \in (T, \infty)$, and that for $\Delta$ full column rank $\Lambda(t) = \{\hat{\boldsymbol{\theta}}\}$ for all $t \in [T, \infty)$.
\end{Remark}

The following section analyzes the robustness of \eqref{eq: FxT Estimation Law} and shows how it may be used to synthesize a safe control law.


\section{Robust Adaptive Safety}\label{sec: Robust Adaptive Safety}
Consider the system \eqref{uncertain system}, now subject to an additional, unknown, possibly time-varying disturbance $d: \mathbb R^n \times \mathbb R_+ \mapsto \mathcal{D} \subset \mathbb{R}^n$ that captures unmodeled effects and is known to satisfy $\sup_{t \geq 0}\|d(\boldsymbol{x}(t),t)\|_\infty \leq B_D < \infty$:
\begin{equation}\label{eq.perturbed_nonlinear_system}
    \dot{\boldsymbol{x}} = f(\boldsymbol{x}) + g(\boldsymbol{x})\boldsymbol{u} + \Delta(\boldsymbol{x})\boldsymbol{\theta}^* + d(\boldsymbol{x},t).
\end{equation}
When using rate measurements as in Remark \ref{rmk.schemes}, it is not possible to obtain a relation of the form \eqref{eq.es_matrix} from \eqref{eq.perturbed_nonlinear_system},
i.e., 
\begin{align*}
    \Delta(\boldsymbol{x}) \Tilde{\boldsymbol{\theta}} = \underbrace{\dot{\boldsymbol{x}} - f(\boldsymbol{x}) - g(\boldsymbol{x})\boldsymbol{u} - \Delta(\boldsymbol{x})\hat{\boldsymbol{\theta}}}_\text{Known} -\underbrace{d(\boldsymbol{x}, t)}_\text{Unknown},
\end{align*}
which instead takes the form \eqref{eq.perturbed_mv}. As such, consider the following.
\begin{Assumption}\label{ass: robust MV}
    For all $t \geq 0$ there is a known vector, $\boldsymbol{v}(t) \in \mathbb R^n$, a known matrix, $\boldsymbol{M}(t) \in \mathbb R^{n \times p}$, and an unknown vector, $\boldsymbol{\delta}(t)$ that jointly satisfy the following properties:
    \vspace{-3mm}
    \begin{enumerate}[label=(\roman*)]
        \item the parameter error, $\Tilde{\boldsymbol{\theta}}(t)$, is one solution to \eqref{eq.perturbed_mv},
        \item there exists known $\Upsilon < \infty$ such that $\sup_{t \in \mathbb R_{\geq 0}}\|\delta(t)\| \leq \Upsilon$
        \item and items (ii) and (iii) of Assumption \ref{ass: v=M0} hold.
    \end{enumerate}
    \vspace{-3mm}
\end{Assumption}
Part (ii) is reasonable given that $d$ in \eqref{eq.perturbed_nonlinear_system} is bounded, i.e., for $\delta(t) = d(\boldsymbol{x}(t),t)$, $\Upsilon = B_D$. Observe also that the above Assumption is valid for systems of the form \eqref{uncertain system} subject to noisy measurements, i.e., 
\begin{align*}
    \Delta(\boldsymbol{y}) \Tilde{\boldsymbol{\theta}} = \underbrace{\dot{\boldsymbol{y}} - f(\boldsymbol{y}) - g(\boldsymbol{y})\boldsymbol{u} - \Delta(\boldsymbol{y})\hat{\boldsymbol{\theta}}}_\text{Known} -\underbrace{\delta(t)}_\text{Unknown},
\end{align*}
where $\boldsymbol{y}$ and $\dot{\boldsymbol{y}}$ denote the measurements of $\boldsymbol{x}$ and $\dot{\boldsymbol{x}}$, and $\delta(t) = (\dot{\boldsymbol{y}} - \dot{\boldsymbol{x}}) + (f(\boldsymbol{y}) -f(\boldsymbol{x})) + (g(\boldsymbol{y}) - g(\boldsymbol{x}))\boldsymbol{u} + (\Delta(\boldsymbol{y}) - \Delta(\boldsymbol{x}))\hat{\boldsymbol{\theta}} + (\Delta(\boldsymbol{y}) - \Delta(\boldsymbol{x}))\Tilde{\boldsymbol{\theta}} + d(\boldsymbol{x}, t)$, provided that a bound on $\boldsymbol{x} - \boldsymbol{y}$ and $\dot{\boldsymbol{x}} - \dot{\boldsymbol{y}}$ is known. For classes of bounded-error observers (e.g., \cite{pylorof2019design,jaulin2002nonlinear}), this is reasonable.

The following Lemma is a requirement for the proof of the subsequent result on the robustness of adaptation law \eqref{eq: FxT Estimation Law}.
\begin{Lemma}\label{lem.proof_function}
    Suppose that $\boldsymbol{x},\boldsymbol{y} \in \mathbb R^n$, and define $P(\boldsymbol{x},\boldsymbol{y}) \triangleq \boldsymbol{x}^T\boldsymbol{y}\left(a\|\boldsymbol{x}+\boldsymbol{y}\|^{\frac{2}{m}} + b\|\boldsymbol{x}+\boldsymbol{y}\|^{\frac{-2}{m}}\right)$ for $a,b>0$, $m>2$.
    If $\exists B_y>0$ such that $\|\boldsymbol{y}\|\leq B_y$ and $\|\boldsymbol{x}\|> 2B_y$, then $P(\boldsymbol{x},\boldsymbol{y}) \geq  -B_y\left(a\|\boldsymbol{x}\|^{1+\frac{2}{m}} + 2^{\frac{2}{m}}b\|\boldsymbol{x}\|^{1-\frac{2}{m}}\right)$.
\end{Lemma}
\vspace{-7mm}
\begin{proof}
    Omitted due to space.
\end{proof}
\begin{Theorem}\label{thm.robust_fxts}
    Consider a dynamical system of the form \eqref{eq.perturbed_nonlinear_system}. Suppose that Assumption \ref{ass: robust MV} holds, and that $\mathcal{N}(\Delta(\boldsymbol{x}(t))) = \mathcal{N}(\Delta(\boldsymbol{x}(0)))$, $\forall t \in [0,T]$. Then, under the adaptation law \eqref{eq: FxT Estimation Law}, with $\lambda_{min}(\boldsymbol{\Gamma}) \geq 2(\frac{\Upsilon}{\underaccent{\bar}{\sigma}(\boldsymbol{M},T)})^2$, there exist neighborhoods $D_0$ and $D$ of $\Omega$ such that for all $\hat{\boldsymbol{\theta}}(0) \in D_0$, the trajectories of \eqref{eq.err_tilde_R} satisfy $\Tilde{\boldsymbol{\theta}}_R(t) \in D_0$ for all $t\geq 0$, and reach $D$ within a fixed time $T$, where $D = \{\Tilde{\boldsymbol{\theta}}_R \; | \; V(\Tilde{\boldsymbol{\theta}}_R) \leq 1\}$, and
    \begin{align}
        D_0 & = \begin{cases}\mathbb R^p; & \Upsilon < Y, \\
        \left\{\Tilde{\boldsymbol{\theta}}_R\; |\; V(\Tilde{\boldsymbol{\theta}}_R)\leq \left(k\frac{\alpha_3-\sqrt{\alpha_3^2-4\alpha_1\alpha_2}}{2\alpha_1}\right)^{\mu}\right\}; & \Upsilon \geq Y,
        \end{cases}\label{eq.robust_fxts_D0}\\
        T & \leq \begin{cases}\frac{\mu}{\alpha_1k_1}\left(\frac{\pi}{2}-\tan^{-1}k_2\right); & \Upsilon<Y, \\
        \frac{\mu k}{(1-k)\sqrt{\alpha_1\alpha_2}}; & \Upsilon \geq Y,
        \end{cases}\label{eq.robust_fxts_T}
    \end{align}
    where $\mu > 2$, $0<k<1$, $k_1 = \sqrt{(4\alpha_1\alpha_2-\alpha_3^2)/4c_1^2}$, and $k_2 = (2\alpha_1-\alpha_3)/\sqrt{4\alpha_1\alpha_2 - \alpha_3^2}$, with
    \begin{align}
        \alpha_1 &= 2^{\frac{-2}{\mu}}ak_V^{2+\frac{2}{\mu}},  \label{eq.a1_robust_adaptation} \\
        \alpha_2 &= 2^{\frac{2}{\mu}}bk_V^{2 - \frac{2}{\mu}},  \label{eq.a2_robust_adaptation} \\
        \alpha_3 &= a\Upsilon k_V^{1+\frac{2}{\mu}} + 2^{\frac{2}{\mu}}b\Upsilon k_V^{1-\frac{2}{\mu}}, \label{eq.a3_robust_adaptation}
    \end{align}
    with $k_V$ defined by \eqref{eq: kV},
    and 
    \begin{align}
        Y &= 2\frac{k_V^2\sqrt{ab}}{ak_V^{1+\frac{2}{\mu}} + 2^{\frac{2}{\mu}}b k_V^{1-\frac{2}{\mu}}}, \label{eq.Y_robust_adaptation}
    \end{align}
\end{Theorem}
\begin{proof}
    Consider the Lyapunov function candidate $V = \frac{1}{2} \Tilde{\boldsymbol{\theta}}_R^T \boldsymbol{\Gamma}^{-1} \Tilde{\boldsymbol{\theta}}_R$. Let $\boldsymbol{q} = \boldsymbol{M}\Tilde{\boldsymbol{\theta}}_R$. Under \eqref{eq: FxT Estimation Law}, $\dot{V} = -\boldsymbol{q}^T\left(\boldsymbol{q} + \boldsymbol{\delta}\right)\\ \left(a\|\boldsymbol{q} + \boldsymbol{\delta}\|^{\frac{2}{\mu}} + b\|\boldsymbol{q} + \boldsymbol{\delta}\|^{-\frac{2}{\mu}}\right)$,
    where $a,b>0$ and $\mu > 2$. By the vector inner product
    $\boldsymbol{q}^T\boldsymbol{q} > |\boldsymbol{q}^T\boldsymbol{\delta}|$
    as long as $\|\boldsymbol{q}\| > \|\boldsymbol{\delta}\|$. Thus, when $\|\boldsymbol{q}\| > \|\boldsymbol{\delta}\|$ it holds that $\boldsymbol{q}^T(\boldsymbol{q} + \boldsymbol{\delta}) > 0$ and therefore that $\dot{V} < 0$. By Assumption \ref{ass: robust MV} and the vector triangle inequality $\|\boldsymbol{q}+\boldsymbol{\delta}\|>\Upsilon$ whenever $\|\boldsymbol{q}\| > 2\Upsilon$. 
    Then, $\|\boldsymbol{q}+\boldsymbol{\delta}\|>\frac{1}{2}\|\boldsymbol{q}\|$ for all $\|\boldsymbol{q}\|>2\Upsilon$. As such, from Theorem \ref{Thm: FxT Parameter Adaptation}, whenever $\|\boldsymbol{q}\| > 2\Upsilon$ the following holds:
    \begin{align}
        \dot{V} \leq &-\boldsymbol{q}^T\boldsymbol{q}\left(a\big(\|\boldsymbol{q}\|/2\big)^{\frac{2}{\mu}} + b\big(\|\boldsymbol{q}\|/2\big)^{-\frac{2}{\mu}}\right) - F(\boldsymbol{q},\boldsymbol{\delta}) \nonumber \\
        \leq &-\alpha_1V^{1+\frac{1}{\mu}} - \alpha_2V^{1-\frac{1}{\mu}} - F(\boldsymbol{q},\boldsymbol{\delta}) < 0, \nonumber
    \end{align}
    where $\alpha_1$ and $\alpha_2$ are given by \eqref{eq.a1_robust_adaptation} and \eqref{eq.a2_robust_adaptation}, and $F(\boldsymbol{q},\boldsymbol{\delta}) =  \boldsymbol{q}^T\boldsymbol{\delta}\left(a\|\boldsymbol{q}+\boldsymbol{\delta}\|^{\frac{2}{\mu}} + b\|\boldsymbol{q}+\boldsymbol{\delta}\|^{-\frac{2}{\mu}}\right)$. Using Lemma \ref{lem.proof_function}, $F(\boldsymbol{q},\boldsymbol{\delta}) \geq-\Upsilon \left(a\|\boldsymbol{q}\|^{1+\frac{2}{\mu}}+2^{\frac{2}{\mu}}b\|\boldsymbol{q}\|^{1-\frac{2}{\mu}}\right)$,
    and therefore
    \begin{align}
        \dot{V} &\leq -\alpha_1V^{1+\frac{1}{\mu}} - \alpha_2V^{1-\frac{1}{\mu}} + a\Upsilon\|\boldsymbol{q}\|^{1+\frac{2}{\mu}} + 2^{\frac{2}{\mu}}b\Upsilon\|\boldsymbol{q}\|^{1-\frac{2}{\mu}}, \nonumber \\
        &\leq -\alpha_1V^{1+\frac{1}{\mu}} - \alpha_2V^{1-\frac{1}{\mu}} + \beta_1V^{\frac{1}{2} + \frac{1}{\mu}} + \beta_2V^{\frac{1}{2} - \frac{1}{\mu}}, \nonumber
    \end{align}
    where $\beta_1=a\Upsilon k_V^{1+\frac{2}{\mu}}$ and $\beta_2=2^{\frac{2}{\mu}}b\Upsilon k_V^{1-\frac{2}{\mu}}$, and which, for all $V\geq 1$, obeys
    \begin{equation}
        \dot{V} \leq -\alpha_1V^{1+\frac{1}{\mu}} - \alpha_2V^{1-\frac{1}{\mu}} + \alpha_3V, \label{eq.robust_V_proof}
    \end{equation}
    where $\alpha_3 = \beta_1+\beta_2$.  Now, since $\frac{1}{2}\lambda_{max}^{-1}(\boldsymbol{\Gamma})\|\Tilde{\boldsymbol{\theta}}\|^2\leq V$, and $\underaccent{\bar}{\sigma}^2(\boldsymbol{M},T)\|\Tilde{\boldsymbol{\theta}}\|^2\leq \|\boldsymbol{q}\|^2$, it is true that $V(\Tilde{\boldsymbol{\theta}}_R) \leq \frac{1}{2}g_1\|\boldsymbol{q}\|^2=V_{max}(\Tilde{\boldsymbol{\theta}}_R)$,
    where 
    $g_1=(\lambda_{min}(\boldsymbol{\Gamma})\underaccent{\bar}{\sigma}^2(\boldsymbol{M},T))^{-1}$. Then, since \eqref{eq.robust_V_proof} holds for all $V \geq 1$ and it has been shown that $\dot{V}<0$ for all $\|\boldsymbol{q}\| > 2\Upsilon$, it follows that $\dot{V}<0$ for all $V> 1$ if $V_{max}(\Tilde{\boldsymbol{\theta}}_R) \leq 1$ when $\|\boldsymbol{q}\| = 2\Upsilon$. Thus, it must hold that $\lambda_{min}(\boldsymbol{\Gamma}) \geq 2(\Upsilon/\underaccent{\bar}{\sigma}(\boldsymbol{M},T))^2$.
    
    
    Since $\forall\|\boldsymbol{q}\| > 2\Upsilon$, both $\dot{V} < 0$ and \eqref{eq.robust_V_proof} hold, it follows from \cite[Lemma 2]{Garg2021Characterization} and \cite[Theorem 1]{Garg2021Characterization} that $\forall \Tilde{\boldsymbol{\theta}}_R(0) \in D_0$, the trajectories of \eqref{eq.err_tilde_R} satisfy $\Tilde{\boldsymbol{\theta}}_R(t) \in D_0$ for all $t\geq 0$, and reach the set $D$ within a fixed time $T$.
    This completes the proof. 
\end{proof}

In what follows, the robust, adaptive CBF condition for forward-invariance of set $S$ is formalized.
\begin{Theorem}\label{Thm: Adaptive Safety}
    Consider a dynamical system of the form \eqref{uncertain system}, and the set $S$ defined by \eqref{eq.safe_set}. Under the premises of Corollary \ref{Cor: upper bound parameter error}, the set $S$ is rendered forward-invariant if the parameter estimates, $\hat{\boldsymbol{\theta}}$, are adapted via \eqref{eq: FxT Estimation Law} and there exists a control input, $\boldsymbol{u}$, for which the following holds $\forall \boldsymbol{x} \in S$, $\forall t \geq 0$:
    \begin{equation}\label{RaCBF Condition}
            \sup_{\boldsymbol{u} \in \mathcal U}\big[ L_fh(\boldsymbol{x}) + L_gh(\boldsymbol{x})\boldsymbol{u}\big] \geq -\alpha \left(h(\boldsymbol{x}) - m(\boldsymbol{\eta}) \right) + \boldsymbol{r}(t,\hat{\boldsymbol{\theta}})
    \end{equation}
    where $\boldsymbol{r}(t,\hat{\boldsymbol{\theta}}) = \boldsymbol{\nu}(\hat{\boldsymbol{\theta}}) + \dot m(\boldsymbol{\eta})$ consists of a robust term 
    \begin{equation}\label{eq: safety projection}
        \resizebox{.875\hsize}{!}{$\boldsymbol{\nu}(\hat{\boldsymbol{\theta}}) = \Sigma_{i=1}^p  \min\left\{C_i\max\left\{\theta_{min},\hat{\theta}_i - \eta\right\},C_i\min\left\{\hat{\theta}_i - \eta,\theta_{max}\right\}\right\}$},
    \end{equation}
    and an adaptive term 
    \begin{equation}\label{eq: eta dot}
        \dot{\eta}(t) = -c_{1}\sqrt{\frac{ \lambda_{max}(\boldsymbol{\Gamma})}{2}}\left(\sqrt{\frac{c_1}{c_2}}\tan\left(A(t)\right)\right)^{\frac{\mu}{2}-1}\sec^{2}\left(A(t)\right),
    \end{equation}
    where $\eta(t)$ is given by \eqref{eq.err_bound}, and
    $A(t)$ is given by \eqref{eq.A(t)} and $C_i$ denotes the $i^{th}$ column of $L_{\Delta}h(\boldsymbol{x})$ for $i \in \{1,\hdots,p\}$.
\end{Theorem}
\vspace{-5mm}
\begin{proof}
    It is sufficient for forward invariance of $S$ to show that \eqref{RaCBF Condition} implies \eqref{eq: cbf condition}.
    First, replace $\boldsymbol{\vartheta}$ from \eqref{eq: cbf condition} with $\boldsymbol{\eta}(t)$. It follows from $\boldsymbol{\Gamma}$ being diagonal that
    \begin{align}
        \frac{\partial h}{\partial \boldsymbol{x}}\dot{\boldsymbol{x}} &= L_fh(\boldsymbol{x}) + L_gh(\boldsymbol{x})\boldsymbol{u} + L_\Delta h(\boldsymbol{x})\boldsymbol{\theta}^*, \label{eq.hr_dot_x}\\
        \frac{\partial m}{\partial \boldsymbol{\eta}}\dot{\boldsymbol{\eta}} &= -\boldsymbol{\eta}^T\Gamma^{-1}\dot{\boldsymbol{\eta}} = -\mathrm{Tr}(\Gamma^{-1})\eta\dot{\eta},
    \end{align}
    where $\dot{\eta}$ given by \eqref{eq: eta dot} is obtained by differentiating \eqref{eq.err_bound}. Note that since $A(0) < \frac{\pi}{2}$ and $A(t)$ decreases monotonically to zero, $\sec^2(A(t))$ and thus $\dot\eta$ are well-defined for all $t \geq 0$.
    
    Then, observe that $L_\Delta h(\boldsymbol{x})\boldsymbol{\theta}^* = L_\Delta h(\boldsymbol{x})\Tilde{\boldsymbol{\theta}} + L_\Delta h(\boldsymbol{x})\hat{\boldsymbol{\theta}} = L_\Delta h(\boldsymbol{x})\Tilde{\boldsymbol{\theta}}_R + L_\Delta h(\boldsymbol{x})\hat{\boldsymbol{\theta}}$, where $\Tilde{\boldsymbol{\theta}}_R$ is as in \eqref{eq.err_tilde_R}. Thus, \eqref{RaCBF Condition} implies forward invariance of $S$ if $L_\Delta h(\boldsymbol{x})\boldsymbol{\theta}^* \geq L_\Delta h(\boldsymbol{x})\Tilde{\boldsymbol{\theta}}_R + L_\Delta h(\boldsymbol{x})\hat{\boldsymbol{\theta}}$. Next, since \eqref{eq.err_bound} defines an upper bound on the error at time $t$, $\|\Tilde{\boldsymbol{\theta}}_R(t)\|_{\infty} = \eta(t)$, let the set of admissible unknown parameters at time $t$ be denoted as $\Phi(t)$, where $\Phi(t) = \Theta \cap \{\boldsymbol{\theta} \in \mathbb R^p: \; \|\boldsymbol{\theta}_R - \hat{\boldsymbol{\theta}}_R(t)\|_{\infty} \leq \eta(t)\}$,
    with $\boldsymbol{\theta} = \boldsymbol{\theta}_R + \boldsymbol{\theta}_N$, and $\boldsymbol{\theta}_R \in \mathcal{R}(\boldsymbol{M})$, $\boldsymbol{\theta}_N \in \mathcal{N}(\boldsymbol{M})$. By definition, $\Phi(t)$ is compact and convex. As such, consider the following minimization problem:
    \begin{equation*}
        J_{\phi}=\min_{\boldsymbol{\phi} \in \Phi(t)} L_{\Delta}h(x)\boldsymbol{\phi},
    \end{equation*}
    which may be solved by solving $p$ linear programs (LPs):
    \begin{equation*}
        \min_{\underaccent{\bar}{\zeta}_i \leq \phi_i \leq \bar{\zeta}_i} C_i\phi_i, \quad \forall i \in \{1,\hdots,p\},
    \end{equation*}
    for each of which, being constrained LPs over compact domains, there exists a (pointwise-in-time) minimum,
    where $\underaccent{\bar}{\zeta}_i = \max\{\theta_{min},\hat{\theta}_i - \eta\}$ and $\bar{\zeta}_i=\min\{\theta_{max},\hat{\theta}_i + \eta\}$ so that $\phi_i^* = \argmin_{\zeta_i \in \{\underaccent{\bar}{\zeta}_i,\bar\zeta_i\}}C_i\zeta_i$.
    Thus, $\forall C_i \neq 0$ it follows that $\phi_i^* = \underaccent{\bar}{\zeta}_i$ or $\phi_i^* = \bar{\zeta}_i$, and $\phi_i^*$ can take any value for $C_i=0$. As such, $\boldsymbol{\nu} = \sum_{i=1}^pC_i\phi_i^* \leq L_{\Delta}h(\boldsymbol{x})\boldsymbol{\theta}^*$, and so \eqref{RaCBF Condition} implies \eqref{eq: cbf condition}. 
\end{proof}
The satisfaction of \eqref{RaCBF Condition} via control $\boldsymbol{u}$ guarantees forward-invariance of the set $S$ by being robust to parameter estimation error and adaptive to how that error changes in time. Further, \eqref{RaCBF Condition} is affine in $\boldsymbol{u}$, which allows it to be synthesized as a constraint in a QP-based control law.


\section{Numerical Case Studies}\label{sec: simulation}

\subsection{``Shoot the Gap": A Comparative Study}

The first study compares the proposed adaptive control scheme with the following works on the ``Shoot the Gap" scenario, where a point mass must pass through a gap between unsafe regions to reach its goal: \cite{Black2021FixedTime} (BLA), \cite{black2020quadratic} (BLR), \cite{Lopez2020racbf} w/o SMID (LOP), \cite{Lopez2020racbf} w/ SMID (LSM), \cite{Taylor2019aCBF} (TAY), and \cite{Zhao2020robustQP} (ZHA). The system model is
\begin{align}\label{eq: simple dynamics}
    \dot{z} = \begin{bmatrix}
            \dot{x}\\
            \dot{y}
        \end{bmatrix} &=
        \begin{bmatrix}
            1 & 0\\
            0 & 1
        \end{bmatrix}
        \begin{bmatrix}
            u_x\\
            u_y
        \end{bmatrix}
        +
        \Delta(z)
        \begin{bmatrix}
            \theta _1\\
            \theta _2
        \end{bmatrix},
\end{align}\normalsize
where $z=[x \; y]^T$ is the state comprised of the lateral ($x$) and longitudinal ($y$) position coordinates with respect to an inertial frame, $\theta_1$, $\theta_2$ are constant parameters that are unknown but bounded, and $\Delta$ is the regressor. Controller performance is first investigated when the regressor is full-rank, i.e., $\Delta=\Delta_F$, where $\Delta_F(z) =  K_{\Delta}\boldsymbol{I}[1 + \sin^2(2\pi f_1 x), 1 + \cos^2(2\pi f_2 y)]^\top$,
with $K_{\Delta},f_1,f_2>0$, and is then examined under a rank-deficient regressor, $\Delta=\Delta_D$, where $\Delta_D(z) = [-x/2, -x; -x/4, -x/2]$.
\begin{figure*}[!ht]
\hspace*{-0.20em}
\begin{subfigure}{0.34\textwidth}
    \includegraphics[width=1\textwidth,left]{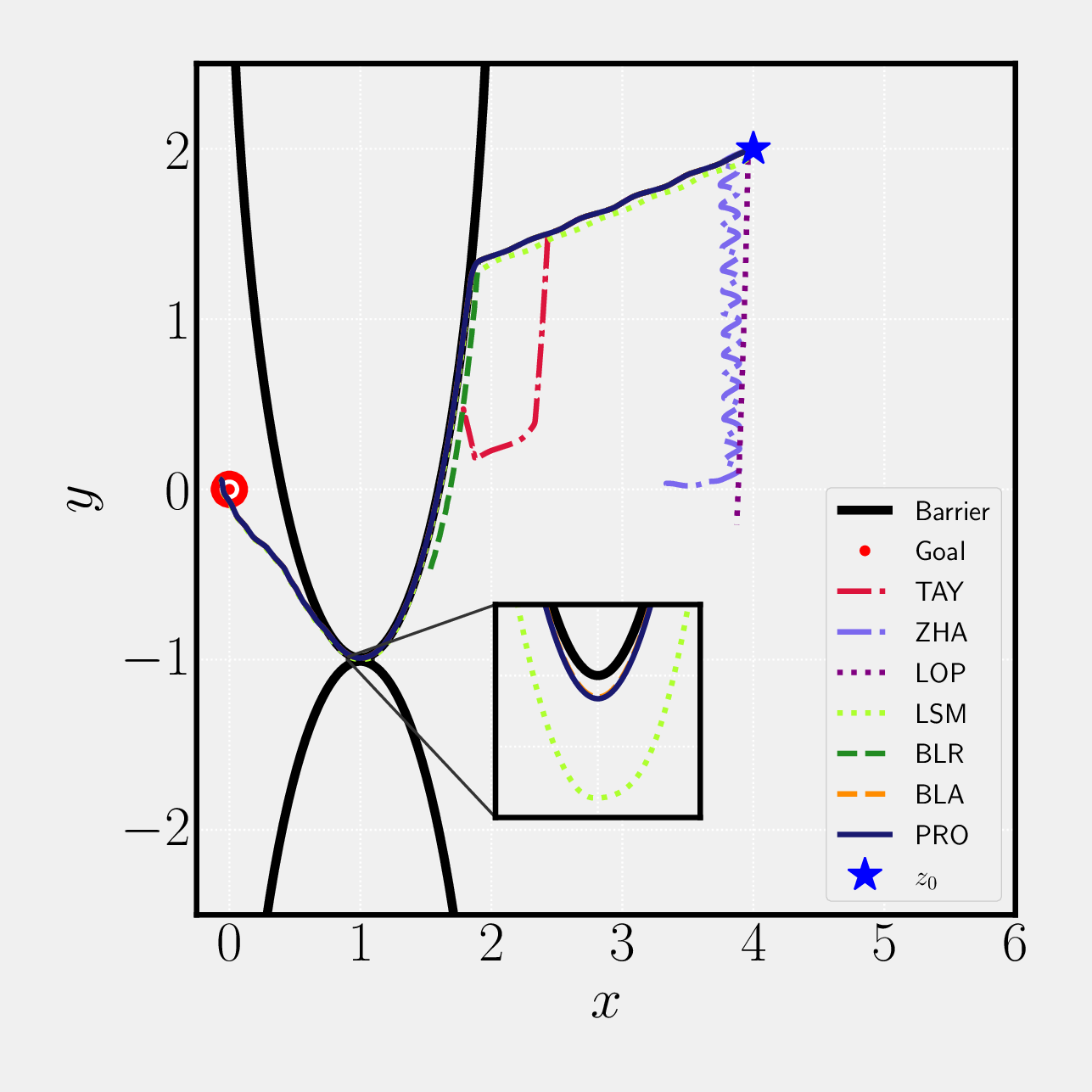}
    \caption{States}\label{fig.simple_trajectories_fr}
\end{subfigure}%
\hspace*{-0.5em}
\begin{subfigure}{0.34\textwidth}
    \includegraphics[width=1\textwidth]{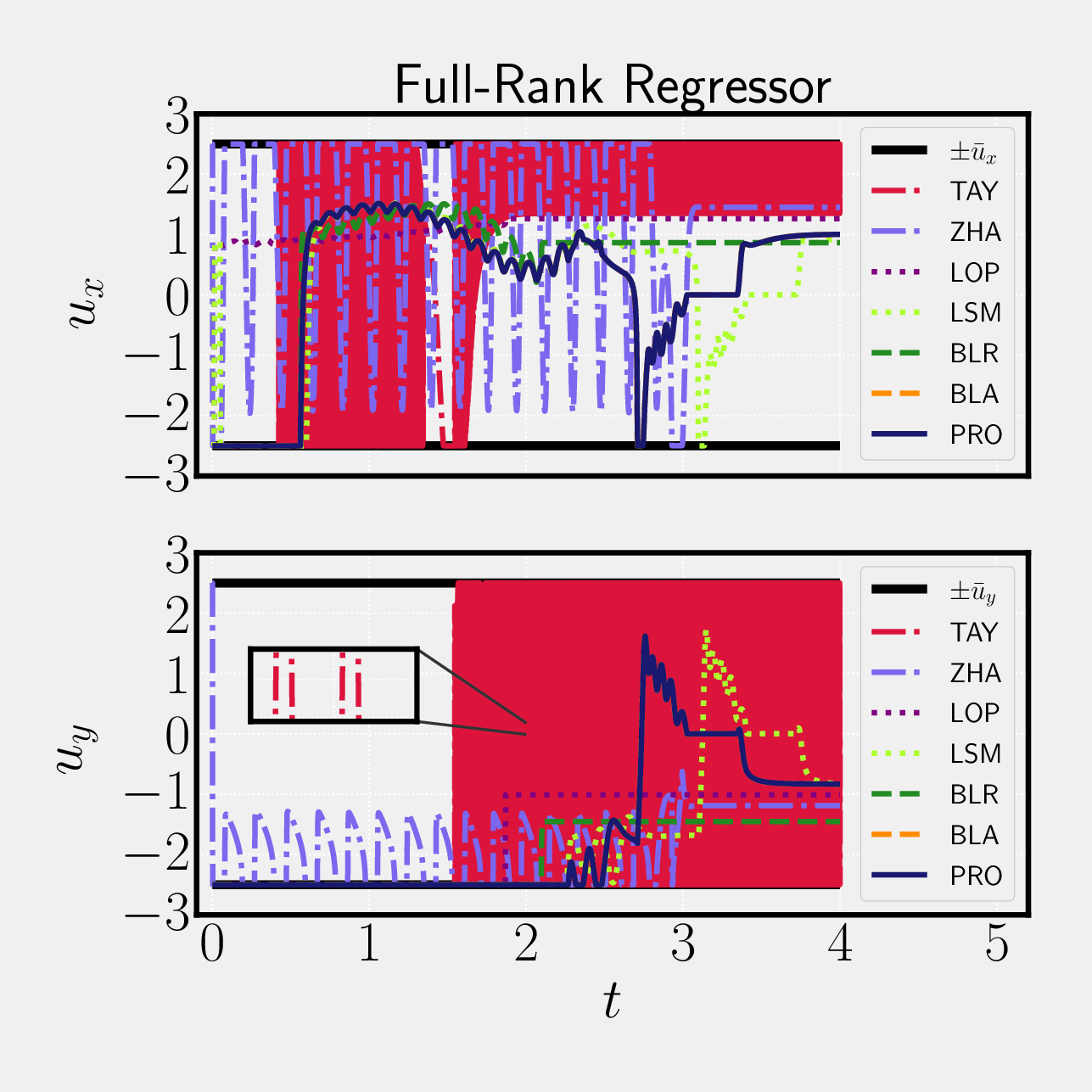}
    \caption{Control inputs}\label{fig.simple_controls_fr}
\end{subfigure}%
\hspace*{-1.25em}
\begin{subfigure}{0.34\textwidth}
    \includegraphics[width=1\textwidth,right]{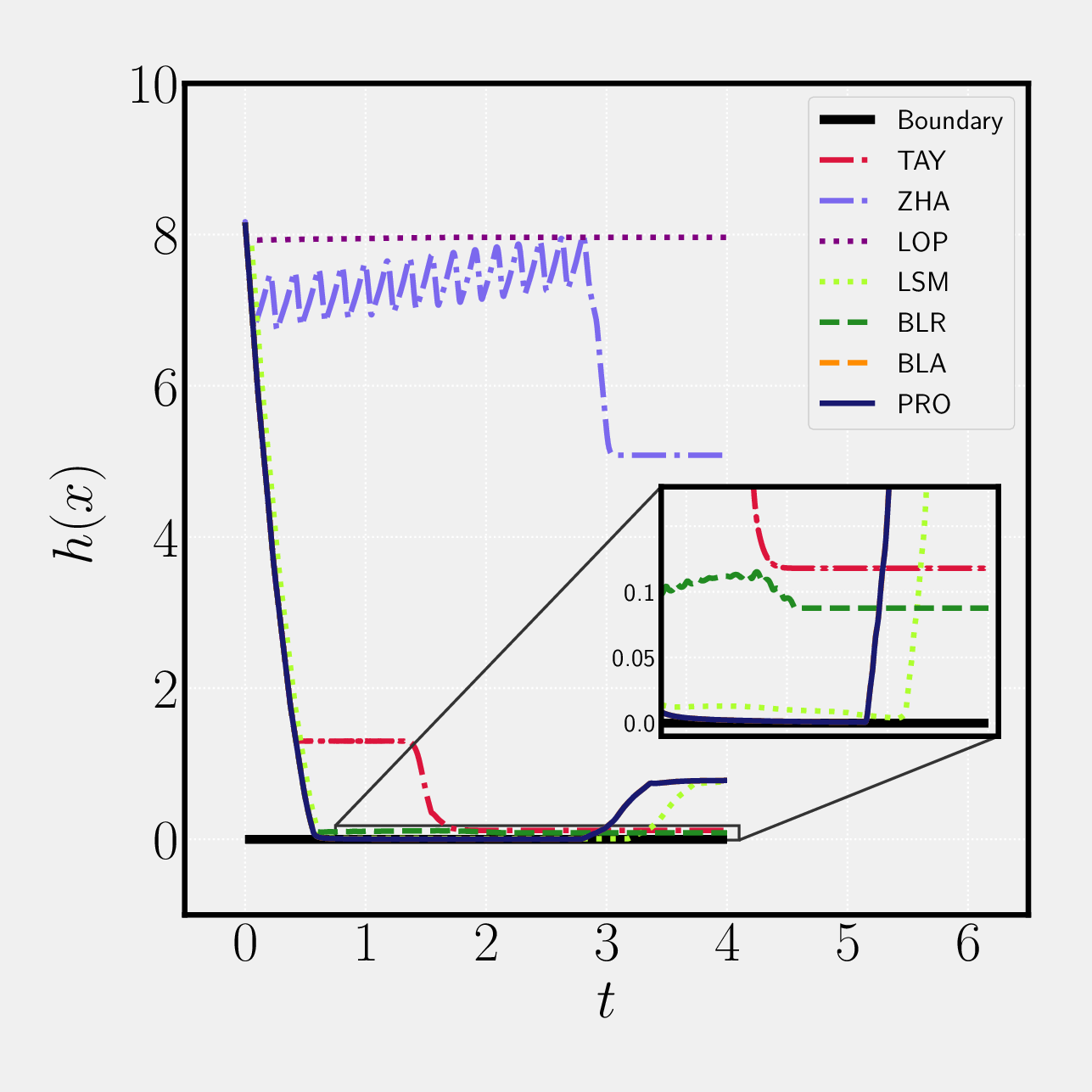}
    \caption{Minimum CBF value}\label{fig.simple_cbfs_fr}
\end{subfigure}%
\caption{\small{Evolutions of the states, control inputs, and control barrier functions for the full-rank regressor ``Shoot the Gap" example.}}
\label{fig.simple_fr}
\end{figure*}
\begin{figure*}[!ht]
\hspace*{-0.20em}
\begin{subfigure}{0.34\textwidth}
    \includegraphics[width=1\textwidth,left]{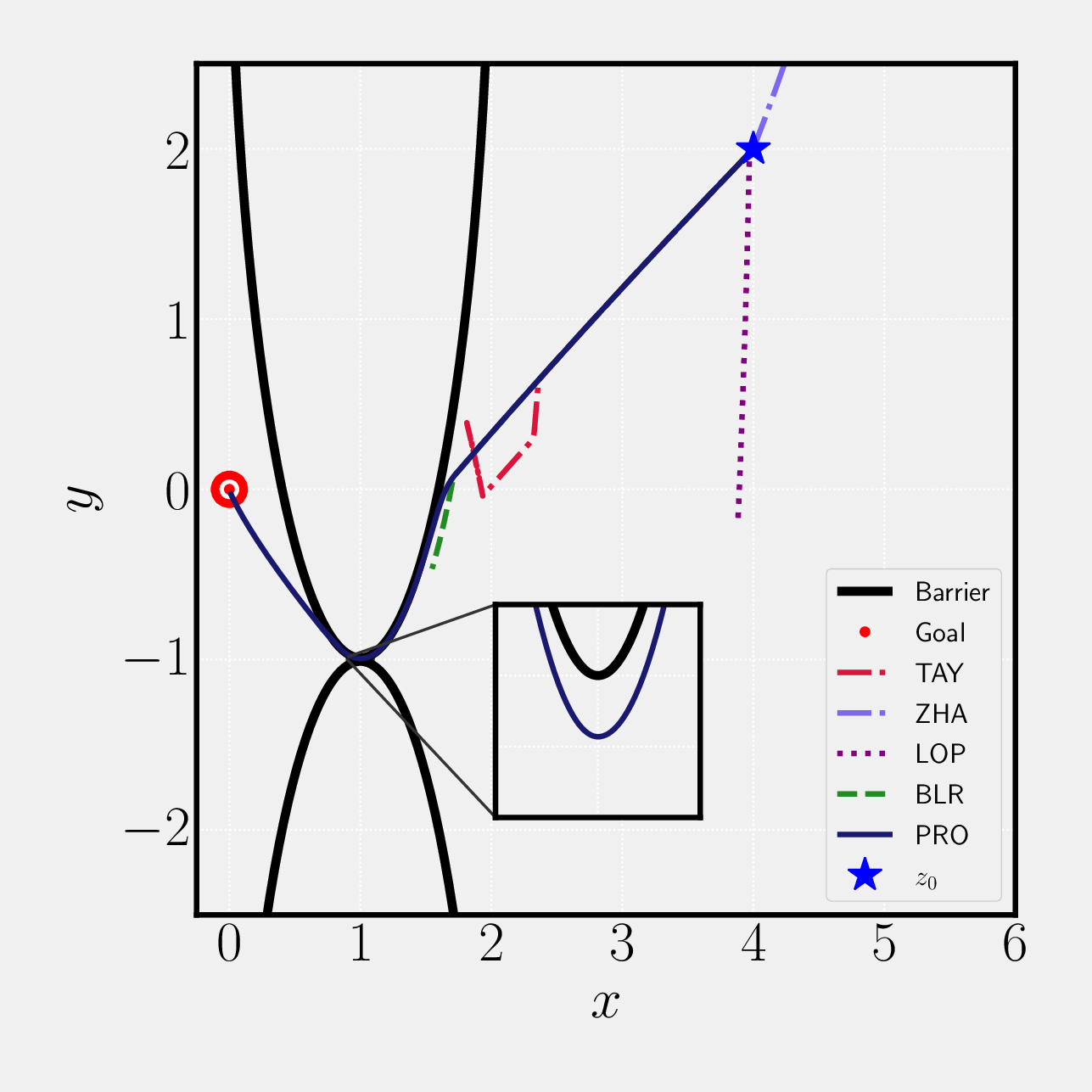}
    \caption{States}\label{fig.simple_trajectories_rd}
\end{subfigure}%
\hspace*{-0.5em}
\begin{subfigure}{0.34\textwidth}
    \includegraphics[width=1\textwidth]{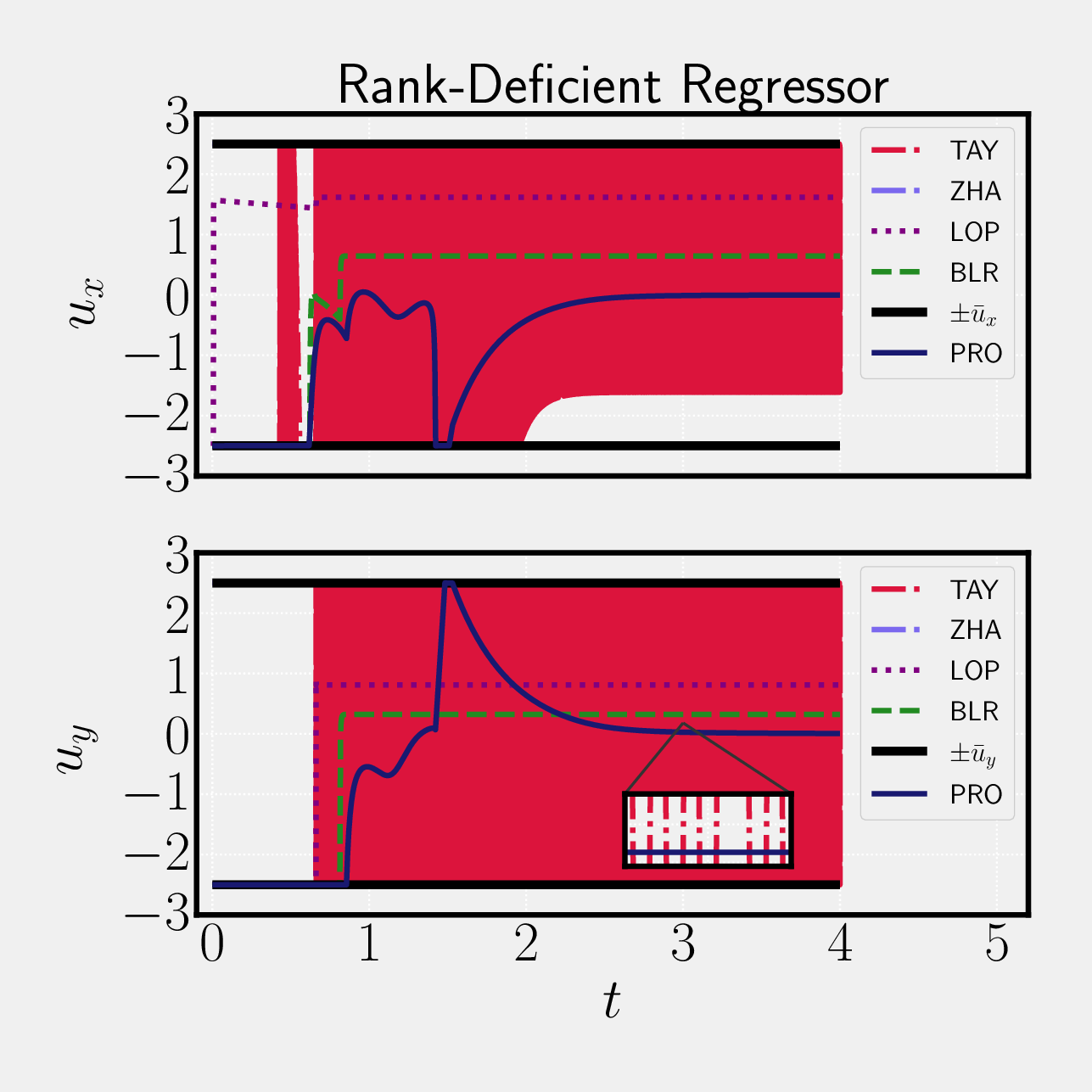}
    \caption{Control inputs}\label{fig.simple_controls_rd}
\end{subfigure}%
\hspace*{-1.25em}
\begin{subfigure}{0.34\textwidth}
    \includegraphics[width=1\textwidth,right]{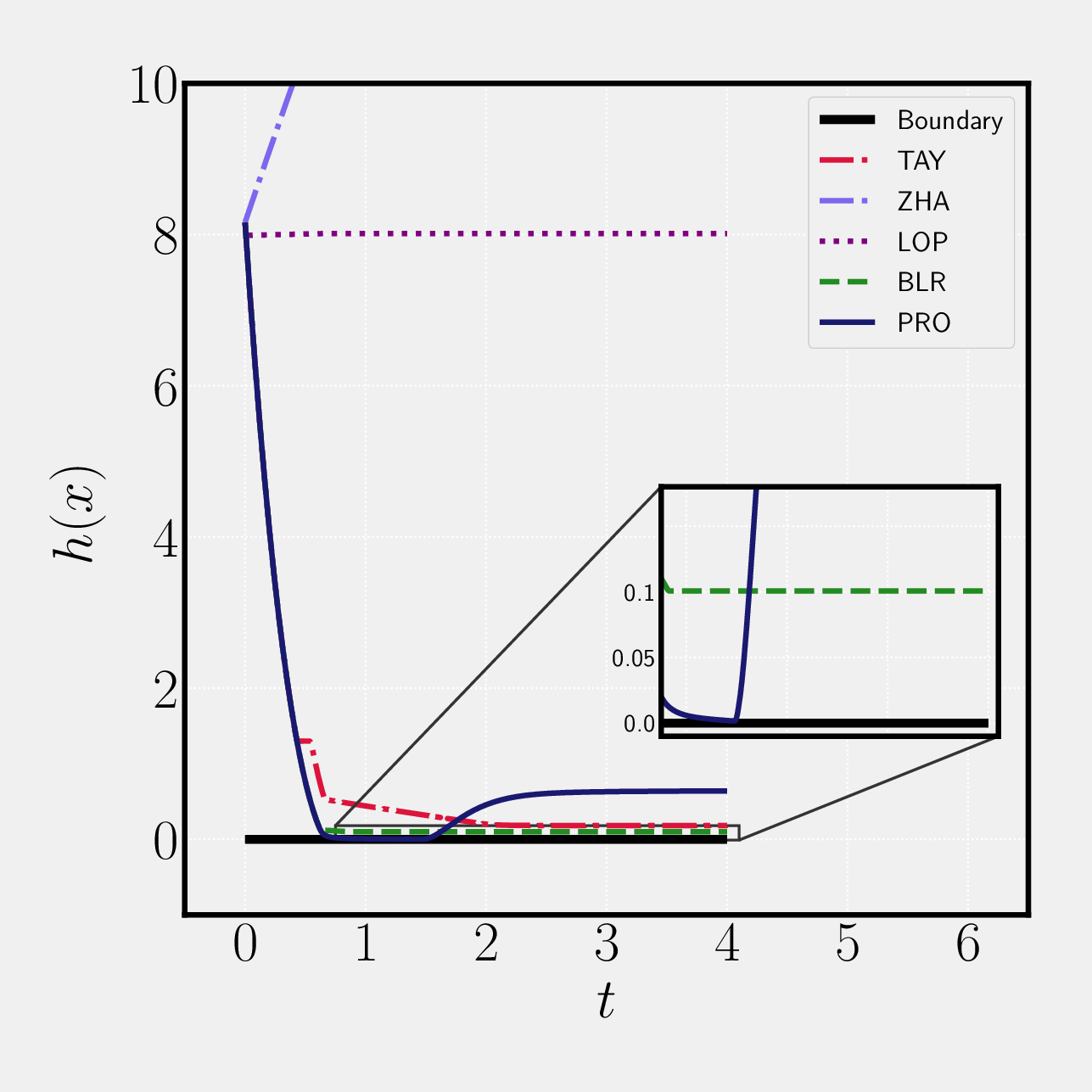}
    \caption{Minimum CBF value}\label{fig.simple_cbfs_rd}
\end{subfigure}%
\caption{\small{Evolutions of the states, control inputs, and control barrier functions for the rank-deficient regressor ``Shoot the Gap" example.}}
\label{fig.simple_rd}
\end{figure*}
For goal-convergence, a control Lyapunov function (CLF) candidate $V(z) = x^2 + y^2$ is used. The safe states are those residing outside of two ellipses shown in Figure \ref{fig.simple_trajectories_fr}, which leads to constraint functions $h_i(z) = \frac{(x - x_i)^2}{a^2} + \frac{(y - y_i)^2}{b^2} - 1$,
for $i \in \{1,2\}$, where $x_1$, $x_2$, $y_1$, $y_2$, $a$, and $b$ define the location, size, and shape of the ellipses.
For control design, the CLF-CBF-QP (\hspace{-0.3pt}\cite{ames2017control,black2020quadratic}) is selected:
\begin{subequations}\label{CLF-CBF QP}
\begin{align}
    \min_{\boldsymbol{u},\delta _0,\delta _1,...,\delta_q} \frac{1}{2}\boldsymbol{u}^{T}Q\boldsymbol{u} &+q_0\delta _0^2 + \sum_{i=1}^q p_i\delta _i^2,\\ 
    \nonumber \textrm{s.t.} \; \forall i \in \{1,&\hdots,q\}, \; \forall j \in \{x,y\} \\
    -\bar{u}_j \leq &u_j \leq \bar{u}_j, \label{u constraint} \\
    1 &\leq \delta_i, \label{d1 constraint} \\
    L_fV+L_gV\boldsymbol{u}&+\phi_V \leq  \delta _0 - c_1V^{\gamma _1} - c_2V^{\gamma _2}, \label{pt_V}\\
    L_fh_i+L_gh_i\boldsymbol{u}&+\phi_h \geq -\delta _ih_i, \label{pt_S1}
\end{align}
\end{subequations}
where $\delta_0$ is a relaxation parameter (penalized by $q_0>0$) that aids the QP feasibility, $\delta_i$ allows for larger negative values of $\dot{h}_i$ away from the boundary of the safe set, and $p_i>0$ penalizes $\delta_i$. The terms $\phi_V = \phi_V(\boldsymbol{x},\Delta(\boldsymbol{x}),\hat{\boldsymbol{\theta}},\eta)$ and $\phi_h = \phi_h(\boldsymbol{x},\Delta(\boldsymbol{x}),\hat{\boldsymbol{\theta}},\eta)$ are left intentionally vague and denote placeholders for how each work from the literature treats the model uncertainty. To summarize, \eqref{u constraint} enforces input constraints, \eqref{d1 constraint} constrains the relaxation coefficients $\delta_i$, \eqref{pt_V} encodes FxT convergence to the goal, and \eqref{pt_S1} guarantees safety.
To fairly compare the proposed approach with the literature, the FxT-CLF condition was tested in all controllers in addition to their native approaches to stabilization. No meaningful differences were found between these cases, and thus results are presented for the FxT-CLF cases.

The simulated state, control, and CBF trajectories for the case of $\Delta=\Delta_F$ are shown in Figure \ref{fig.simple_fr}. The only controllers that safely drive the vehicle to the goal are those that closely estimate the unknown parameters in the system dynamics, i.e., the BLA, LSM, and PRO approaches. It is apparent from Figure \ref{fig.simple_theta_fr} that the parameter estimates converge to their true values within fixed-time $T$ given by \eqref{eq: FxT Time Bound}. 
When $\Delta = \Delta_D$, however, both the BLA and LSM methods that succeed under a full-rank regressor fail; each requires the inversion of a non-invertible matrix. As illustrated by Figure \ref{fig.simple_rd}, the proposed controller succeeds in safely reaching the goal by learning a parameter vector, $\hat{\boldsymbol{\theta}}$, within a fixed time that satisfies $\Delta_D(\boldsymbol{x})\hat{\boldsymbol{\theta}}=\Delta_D(\boldsymbol{x})\boldsymbol{\theta}^*$, as supported by Figure \ref{fig.simple_theta_rd}. Under these dynamics, the proposed approach is the only controller that successfully reaches the goal. In this sense, it is clear that the proposed adaptation law adds value to control for the class of systems \eqref{uncertain system}.


\begin{figure}[!h]
    \centering
        \includegraphics[width=0.85\columnwidth,clip]{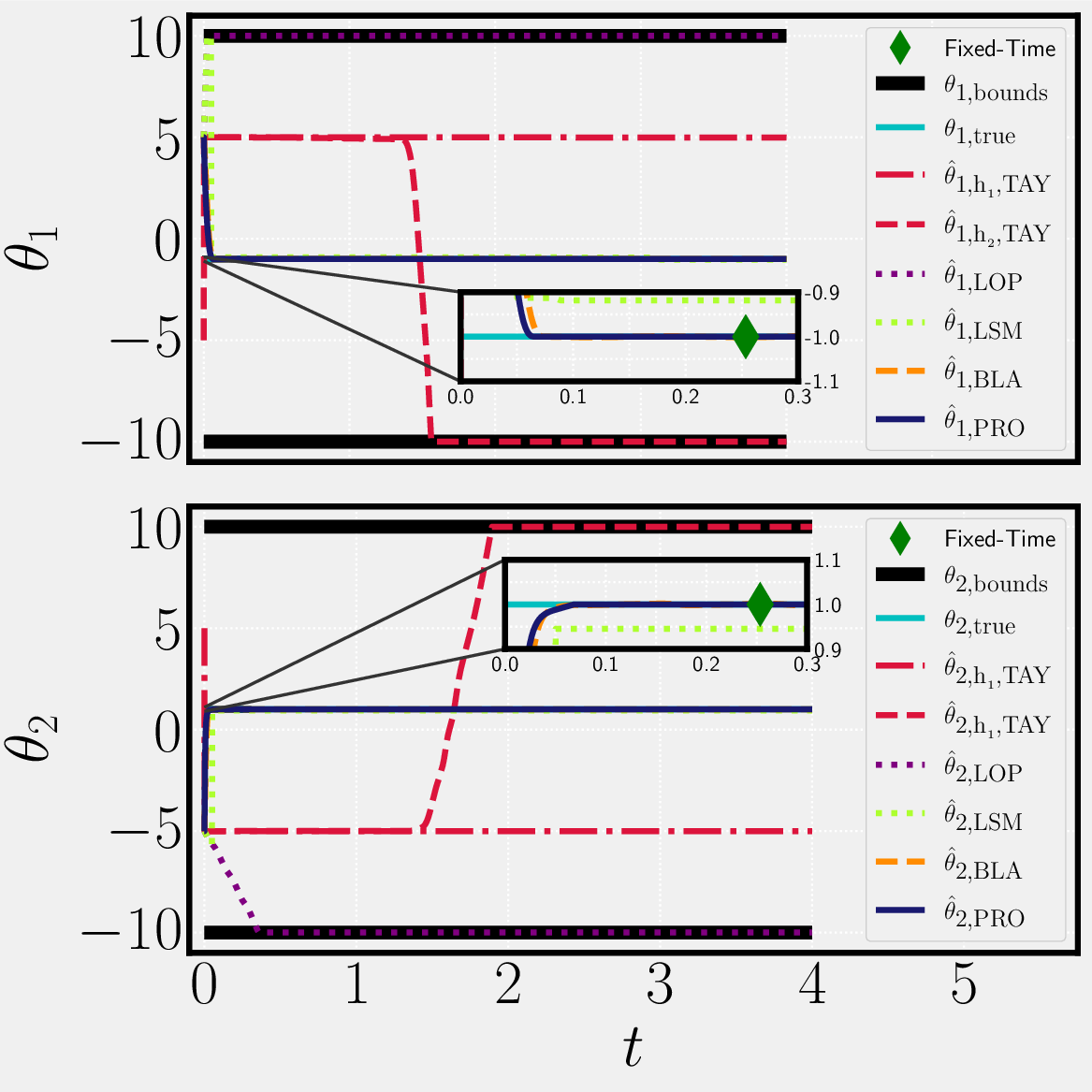}
    \caption{\small{Parameter estimate trajectories for the full-rank regressor ``Shoot the Gap" example.}}\label{fig.simple_theta_fr}
\end{figure}

\begin{figure}[!h]
    \centering
        \includegraphics[width=0.85\columnwidth,clip]{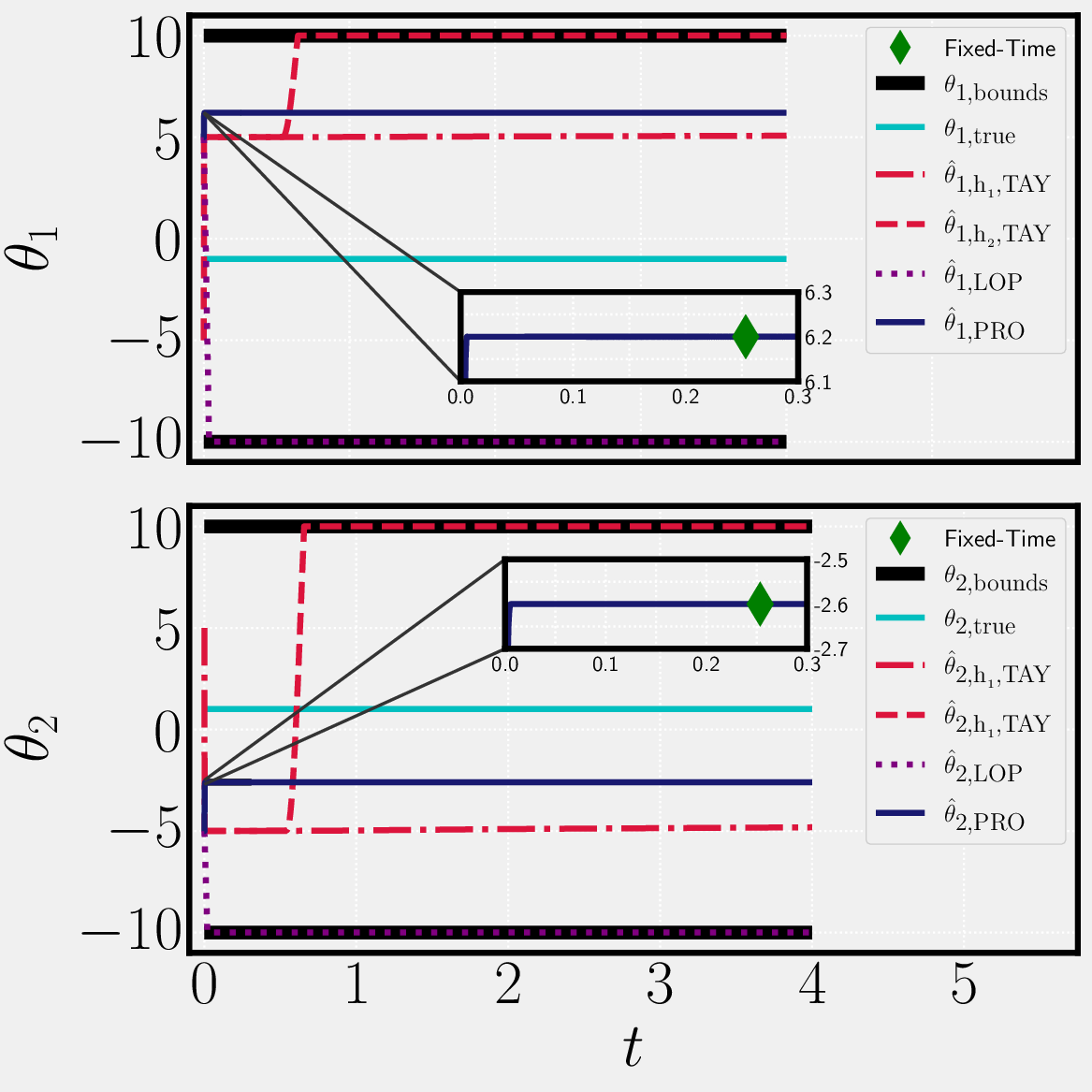}
    \caption{\small{Parameter estimate trajectories for the rank-deficient regressor ``Shoot the Gap" example.}}\label{fig.simple_theta_rd}
\end{figure}

\subsection{Quadrotor with unknown drag coefficients}
The next study examines a simulated quadrotor in a wind field: the objective is to track a Gerono lemniscate (figure-eight) trajectory while obeying attitude and altitude constraints despite wind perturbations. The nominal quadrotor is modeled via the 6-DOF rigid-body dynamic model in \cite[Ch. 3, Eq. (16) - (19)]{Beard2008quadrotordynamics}. The state is $\boldsymbol{\chi} = [x\; y\; z\; u\; v\; w\; \phi\; \theta\; \psi\; p\; q\; r]^T$; where $x$, $y$, and $z$ are the position coordinates (in m) w.r.t. an inertial frame $\mathcal{I}$; $u$, $v$, and $w$ are the translational velocities (in m/s) w.r.t. the body-fixed frame $\mathcal{B}$; $\phi$, $\theta$, and $\psi$ (in rad) are the roll, pitch, and yaw Euler angles defining a ZYX rotation from $\mathcal{I}$ to $\mathcal{B}$; and $p$, $q$, and $r$ are the roll, pitch, and yaw rates (in rad/s) defined w.r.t. $\mathcal{B}$. While wind perturbs the dynamics via such phenomena as blade flapping, induced velocity, and aerodynamic drag (see \cite{Sydney2013quadrotor,tran2015quadrotor}), in this paper only aerodynamic drag is considered.


\subsubsection{Constant Wind Field}
The first scenario is the quadrotor perturbed by a constant, known wind field. Its drag coefficient vector $\boldsymbol{C}_d = [C_x \; C_y \; C_z]^T$ is unknown, where $C_x$, $C_y$, and $C_z$ are the components in $\mathcal{B}$.
The control inputs are thrust of the rotors $F$ (in N) and rotor-induced rolling, pitching, and yawing torques $\tau_\phi$, $\tau_\theta$, and $\tau_\psi$ (in N$\cdot$m) respectively. Since the model is affine in $\boldsymbol{u} = [F \; \tau_{\phi} \; \tau_{\theta} \; \tau_{\psi}]^T$, it can be written as 
$\dot{\boldsymbol{\chi}} = f(\boldsymbol{\chi}) + g(\boldsymbol{\chi})\boldsymbol{u} + \Delta(\boldsymbol{\chi})\boldsymbol{C}_d$, where $\Delta(\boldsymbol{\chi}) = [\boldsymbol{0}_{3 \times 3} \; \Delta_a(\boldsymbol{\chi}) \; \boldsymbol{0}_{6 \times 3}]^T$, with $\Delta_a(\boldsymbol{\chi}) = K_\Delta \frac{\|\boldsymbol{v}_r\|}{M}\boldsymbol{I}[v_{r, 1} \; v_{r, 2} \; v_{r,3}]^\top$,
where $K_\Delta > 0$ and $\boldsymbol{v}_r = \boldsymbol{R}\boldsymbol{v}_w - \boldsymbol{v}_q$ is the relative wind-velocity vector in $\mathcal{B}$, with $\boldsymbol{v}_q = [u \; v \; w]^T$ the quadrotor velocity vector w.r.t. $\mathcal{B}$, $\boldsymbol{v}_w$ the wind velocity vector w.r.t. $\mathcal{I}$, and $\boldsymbol{R}$ the rotation matrix from $\mathcal{I}$ to $\mathcal{B}$. Additionally, $v_{r,1}$, $v_{r,2}$, and $v_{r,3}$ are the principal components of $\boldsymbol{v}_r$
such that $\Delta(\boldsymbol{\chi})\boldsymbol{C}_d$ models the effect of aerodynamic drag acting on the center of mass of the quadrotor. 

The following CBF-QP control law (e.g., \cite{black2020quadratic}) is used:
\begin{subequations}\label{eq: CBF-QP Compensating Controller}
\begin{align}
    \boldsymbol{v}^* = \argmin_{\boldsymbol{v}} \frac{1}{2}\|\boldsymbol{u} &- \boldsymbol{u}_{nom}\|^2 + p_1\delta _1^2 + p_2\delta _2^2 \label{subeq: CBF-QP J}\\
    \nonumber &\textrm{s.t.} \\
    A\boldsymbol{v} &\leq \boldsymbol{b} \label{subeq: control constraint} \\
    L_fh_1(\boldsymbol{\chi})+L_gh_1(\boldsymbol{\chi})\boldsymbol{u} &\geq -\delta _1h_1(\boldsymbol{\chi}) - \boldsymbol{r}(t,\hat{\boldsymbol{C}}_d) \label{subeq: h1 constraint} \\
    L_fh_2(\boldsymbol{\chi})+L_gh_2(\boldsymbol{\chi})\boldsymbol{u} &\geq -\delta _2h_2(\boldsymbol{\chi}) - \boldsymbol{r}(t,\hat{\boldsymbol{C}}_d) \label{subeq: h2 constraint}
\end{align}
\end{subequations}
where $\boldsymbol{v}=[\boldsymbol{u}\; \delta_1\; \delta_2]^T \in \mathbb R^{6}$, and $p_1,p_2>0$. \eqref{subeq: CBF-QP J} seeks to minimize the deviation of $\boldsymbol{u}^*$ from some nominal control $\boldsymbol{u}_{nom}$, and to minimize the slack variables $\delta_1^*,\delta_2^*$ penalized by $p_1,p_2>0$; \eqref{subeq: control constraint} encodes control input constraints and enforces that $\delta_1^*,\delta_2^*\geq \underaccent{\bar}{\delta} > 0$; and \eqref{subeq: h1 constraint} and \eqref{subeq: h2 constraint} enforce constraint functions $h_1(\boldsymbol{\chi}) = 1 - \left((z - c_z)/p_z\right)^{n_z}$ and $h_2(\boldsymbol{\chi}) = \cos(\phi)\cos(\theta) - \cos(\alpha)$,
where $c_z=2.5$, $p_z=2.5$, and $n_z=2$ encode that $h_1$ defines the safe altitude range as $[0,5]$, and $h_2$ restricts the thrust vectoring angle to be within $\alpha=\frac{\pi}{2}$ rad of the vertical. 
For the nominal input $\boldsymbol{u}_{nom}$, the tracking controller developed for quadrotors by \cite{Schoellig2012PeriodicQuadrotor} is used.

The quadrotor was examined under the adaptation law \eqref{eq: FxT Estimation Law}: $\dot{\hat{\boldsymbol{C}}}_d = \boldsymbol{\Gamma}\boldsymbol{W}^T\boldsymbol{e}\left(a\|\boldsymbol{e}\|^{\frac{2}{\mu}} + b\|\boldsymbol{e}\|^{\frac{-2}{\mu}}\right)$,
where $\boldsymbol{e} = \boldsymbol{\chi} - \boldsymbol{z}$, using the predictor scheme mentioned in Remark \ref{rmk.schemes}: $\dot{\boldsymbol{z}} = f(\boldsymbol{\chi}) + g(\boldsymbol{\chi})\boldsymbol{u} + \Delta(\boldsymbol{\chi})\hat{\boldsymbol{C}}_d + k_e\boldsymbol{e} + \boldsymbol{W}\dot{\hat{\boldsymbol{C}}}_d$, $\boldsymbol{z}(0) = \boldsymbol{\chi}(0)$,
where $\dot{\boldsymbol{W}} = -k_e\boldsymbol{W} + \Delta(\boldsymbol{\chi})$, $\boldsymbol{W}(0) = \boldsymbol{0}_{n \times p}$,
and $a,b,\mu = 5$, $k_e=10$. For comparison, the case of zero adaptation ($\dot{\hat{\boldsymbol{C}}}_d=\boldsymbol{0}_{3\times 1}$,  denoted 0-CW) was also studied. The wind-velocity vector $\boldsymbol{v}_w$ was held constant at $\boldsymbol{v}_r = [10, \; -8, \; -5]^T$, which resulted in the regressor being full-rank for all $t \geq 0$. The warm-up time was taken to be $T_w=0.1$ sec. 
Figure \ref{fig: quad theta hats} indicates that the estimates of the coefficients of drag converged to their true values within the fixed-time bound. Figures \ref{Fig: Quadrotor Trajectories Constant Wind} and \ref{fig.quadrotor_cbfs} highlight that the controller safely tracks the reference trajectory using \eqref{eq: FxT Estimation Law} with the state prediction scheme, whereas poor tracking, albeit safe, is achieved without parameter adaptation. Control trajectories are provided in Figure \ref{fig.quadrotor_controls}.

\begin{figure}[!h]
    \centering
        \includegraphics[width=0.85\columnwidth,clip]{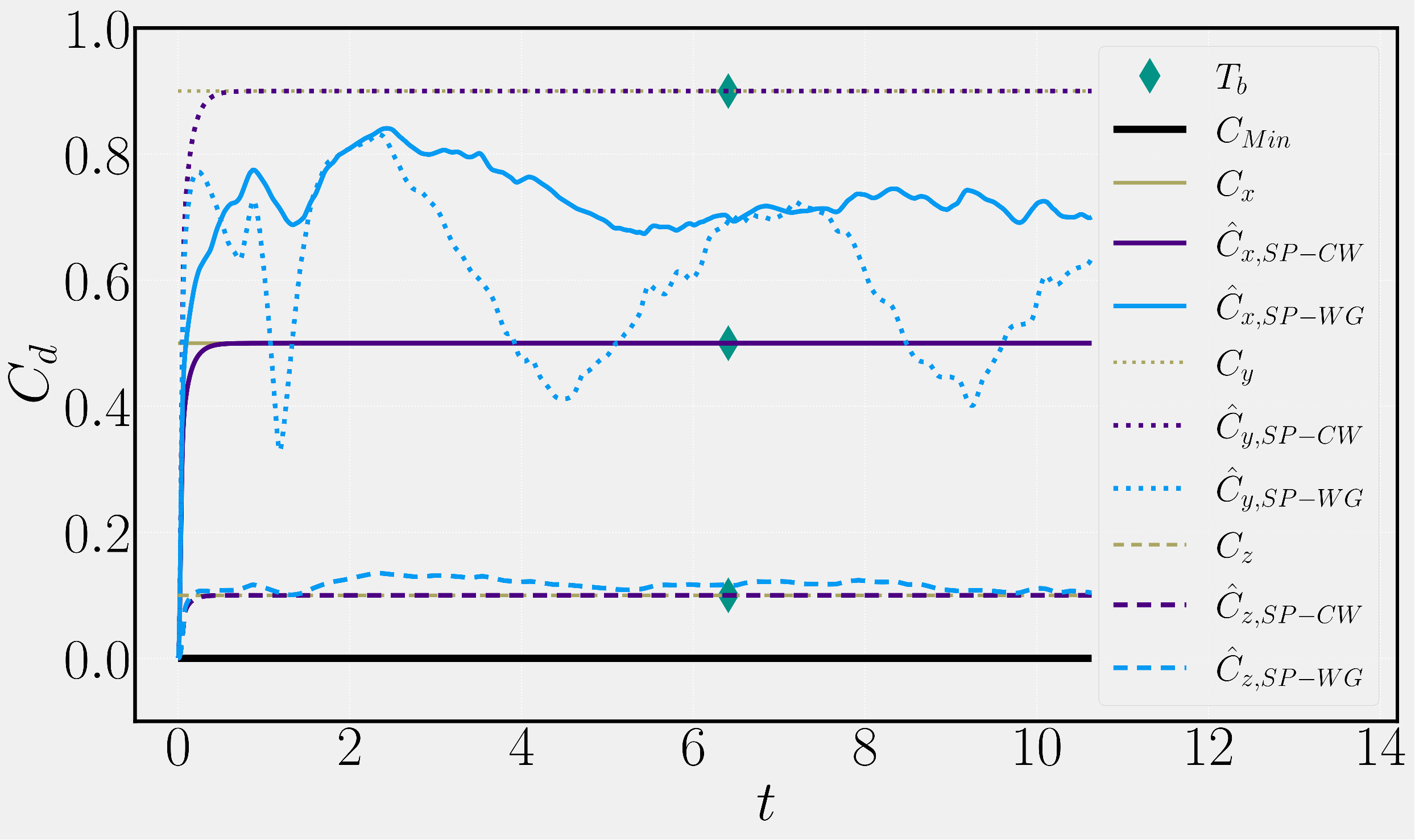}
    \caption{\small{Principal coefficient of drag estimates for the quadrotor in a 1) constant (CW) and 2) gusty (WG) wind field under the proposed controller with the state prediction (SP) scheme.}}\label{fig: quad theta hats}
\end{figure}

\begin{figure}[!ht]
    \centering
        \includegraphics[width=0.85\columnwidth,clip]{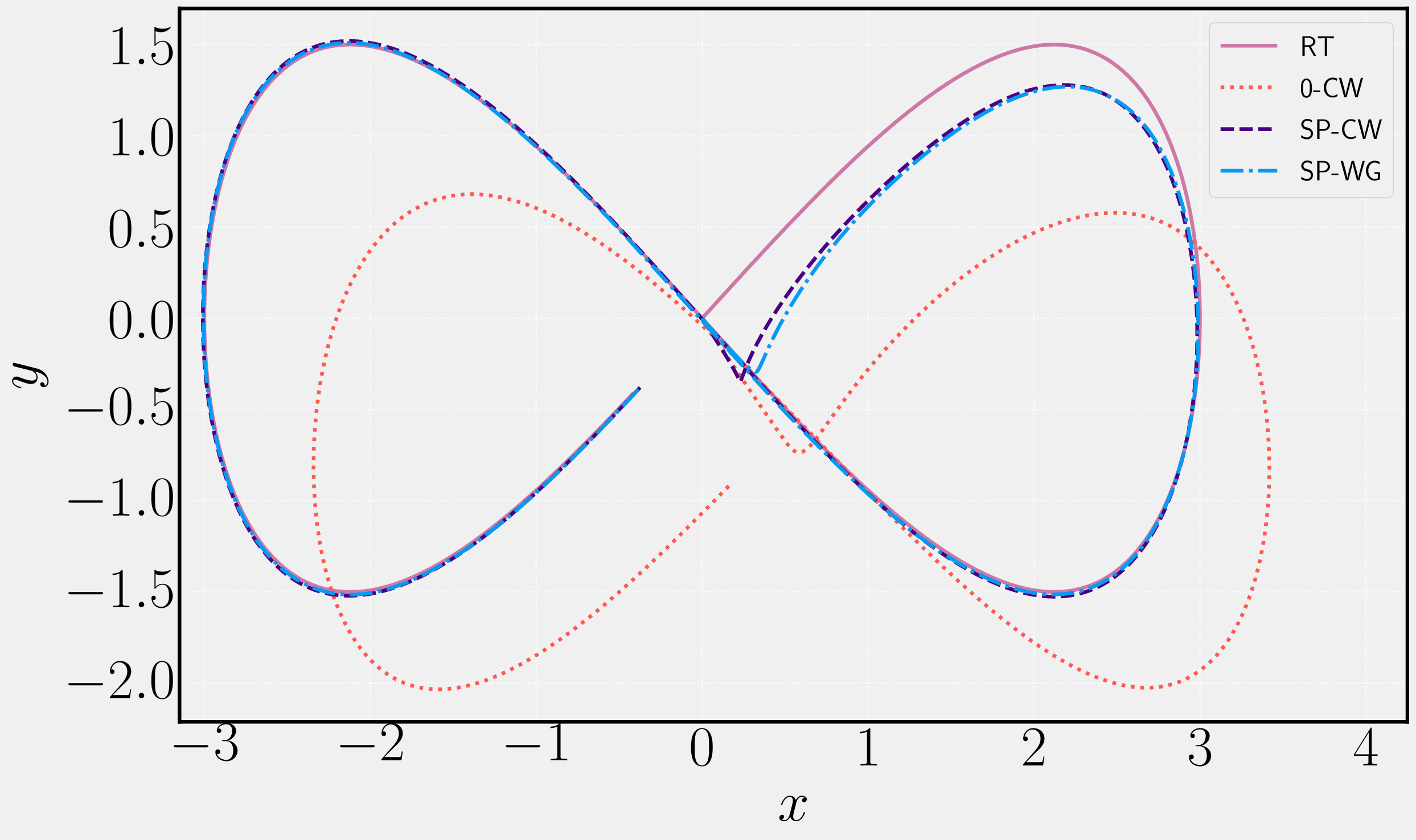}
    \caption{\small{Quadrotor XY trajectories as the controller seeks to track the reference trajectory (RT) in a wind field.}}\label{Fig: Quadrotor Trajectories Constant Wind}
\end{figure}

\begin{figure}[!ht]
    \centering
        \includegraphics[width=0.85\columnwidth,clip]{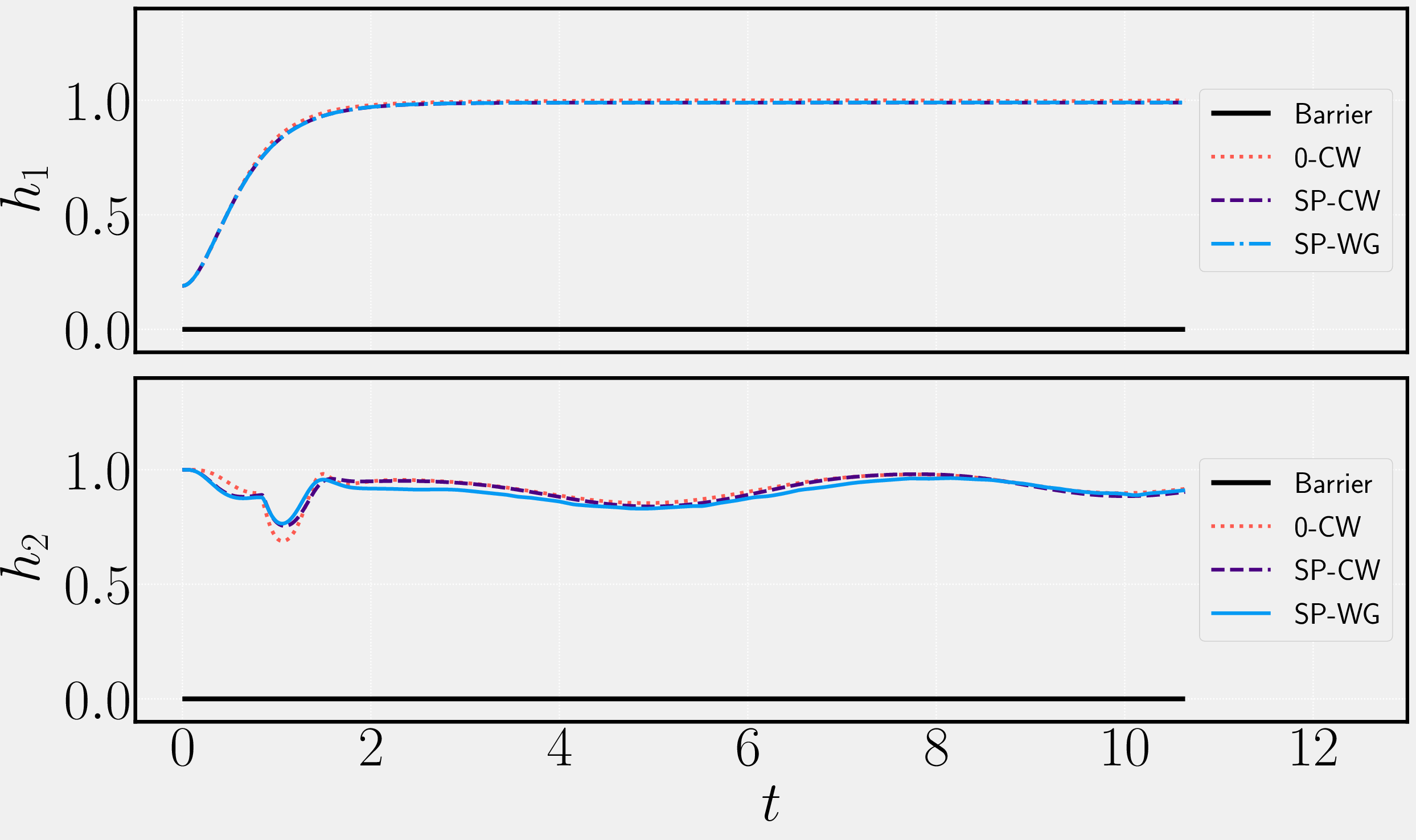}
    \caption{\small{CBF trajectories for the quadrotor numerical study.}}\label{fig.quadrotor_cbfs}
\end{figure}

\begin{figure}[!ht]
    \centering
        \includegraphics[width=0.85\columnwidth,clip]{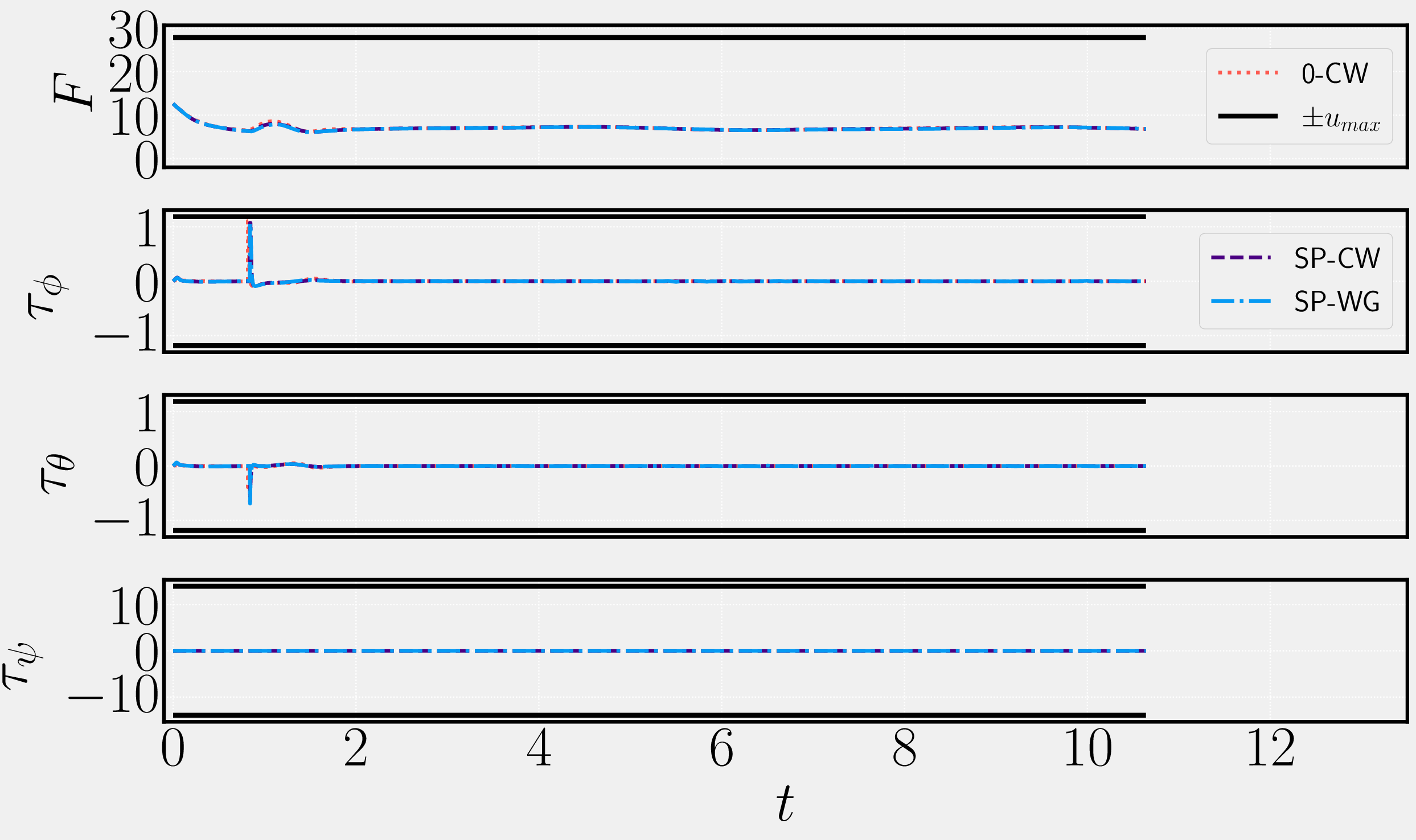}
    \caption{\small{Control inputs for the quadrotor numerical study.}}\label{fig.quadrotor_controls}
\end{figure}

\subsubsection{Unknown Wind Field}
The quadrotor was also simulated in an unknown, gusty wind field (denoted SP-WG). The gusts are an additive disturbance $\boldsymbol{d}(\boldsymbol{\chi}) = [\boldsymbol{0}_{3 \times 3} \; \boldsymbol{d}_a(\boldsymbol{\chi}) \; \boldsymbol{0}_{6 \times 3}]^\top$, where 
$\boldsymbol{d}_a(\boldsymbol{\chi}) = \left( K_\Delta\frac{\|\boldsymbol{v}_r^*\|}{M}\boldsymbol{I}[v_{r,1}^* , v_{r,2}^*, v_{r,3}^*]^\top - \Delta_a(\boldsymbol{\chi})\right)[C_x, C_y, C_z]^\top$ with $\boldsymbol{v}_r^*=\boldsymbol{R}(\boldsymbol{v}_w + \boldsymbol{v}_G) - \boldsymbol{v}_q$ the relative-wind velocity vector considering both the known wind velocity $\boldsymbol{v}_w$ from the prior example and unknown gust velocity $\boldsymbol{v}_G$ generated using the model from \cite{davoudi2020quad}, with principal components $v_{r,1}^*$, $v_{r,2}^*$, and $v_{r,3}^*$. 
Simulation parameters lead to $Y=3.65$ using \eqref{eq.Y_robust_adaptation} and $\Upsilon = 0.69$ from the wind model. 
The choice of $\lambda_{min}(\boldsymbol{\Gamma})=63.16$ satisfies $\lambda_{min}(\boldsymbol{\Gamma}) > 2(\frac{\Upsilon}{\underaccent{\bar}{\sigma}(\boldsymbol{M})})^2 = 26.36$, where $\underaccent{\bar}{\sigma}(\boldsymbol{M}, T - T_w)\geq 0.19$ for $t \geq T_w = 0.1$. Figures \ref{Fig: Quadrotor Trajectories Constant Wind} and \ref{fig.quadrotor_cbfs} show that the controller achieves safe, accurate tracking despite the parameter estimates not converging to the true values. It is believed that this is because the estimates capture the full effect of the uncertainty in the dynamics, i.e., $\Delta(\boldsymbol{\chi})\hat{\boldsymbol{C}}_d \approx \Delta(\boldsymbol{\chi})\boldsymbol{C}_d + d(t,\boldsymbol{\chi})$, though the analysis is left to future work.

\section{Conclusion}\label{sec: conclusion}
In this study on the merits of fixed-time adaptation for safe control under parametric model uncertainty, an adaptation law was introduced to learn additive, parametric uncertainty associated with a class of nonlinear, control-affine systems within a fixed time without requiring rank or persistence of excitation conditions on the regressor. The robustness of the proposed law to a class of bounded disturbances was analyzed, and a robust, adaptive CBF-based control law was synthesized after deriving a time-varying bound on the disturbance estimation error. A comparative study on a 2D single integrator system was conducted, and it was shown that the proposed approach succeeds where others from the literature fail. The value of the approach was further demonstrated on a safe trajectory-tracking problem using a quadrotor dynamical model in wind fields both known and unknown.

In the future, fixed-time adaptation for safe control of systems with nonlinearly parameterized and non-parametric additive disturbances to the system dynamics will be studied.





\begin{ack} 
The authors would like to acknowledge the support of the National Science Foundation award number 1931982.
\end{ack}

\bibliographystyle{plain}        
\bibliography{myreferences}

\end{document}